\documentclass[a4paper]{article}
\usepackage{a4wide}
\usepackage{latexsym}
\usepackage[dvipdfmx]{graphicx}
\usepackage{tikz}
\usepackage[T1]{fontenc}
\def\eps{\varepsilon}

\def\lam{\lambda}
\def\Lam{\Lambda}
\def\del{\delta}
\def\sig{\sigma}
\def\nab{\nabla}
\def\aa{\mathbf{a}}
\def\ep{\mathbb{E}}
\def\prb{\mathbb{P}}
\def\ll{\underline{L}}
\def\hh{\widehat{\underline{H}}}
\def\sq{\square}
\def\ka{\kappa}
\def\Om{\Omega}
\usepackage{amsmath, amssymb}
\usepackage{esint}
\usepackage{mathtools}
\DeclareMathOperator{\esssup}{ess\,\sup}
\DeclareMathOperator{\essinf}{ess\,\inf}
\mathtoolsset{showonlyrefs=true}
\usepackage{amsthm}
\theoremstyle{plain}
\newtheorem{thm}{Theorem}[section]
\newtheorem{lem}[thm]{Lemma}
\newtheorem{prop}[thm]{Proposition}
\newtheorem{dfn}[thm]{Definition}

\newtheorem{rem}[thm]{Remark}

\numberwithin{equation}{section}
\usepackage{type1cm}

\begin{document}

\renewcommand{\thefootnote}{\fnsymbol{footnote}}

\title{Quantitative Stochastic Homogenization of Elliptic Equations with Unbounded Coefficients}
\author{Tomohiro Aya\footnote{Department of Mathematics, Graduate School of Science, Kyoto University, Kitashirakawa-Oiwakecho, Sakyo-ku,
Kyoto 606-8502, Japan. E-mail: aya.tomohiro.42z@st.kyoto-u.ac.jp}}
\date{2023/02/02}
\maketitle

\section*{Abstract}
In this paper, we consider stochastic homogenization of elliptic equations with unbounded and non-uniformly elliptic coefficients. Extending subadditive arguments which are introduced in \cite{MR3932093,MR3481355}, we get an estimate for the rate of the convergence of the solution of the Dirichlet problem under the condition that coefficients in the unit cube have a certain exponential integrability. For the coefficient field $\aa$ in this paper, we only assume a constant decrease at a constant distance of the maximal correlation as an assumption of ergodicity, and stationarity with respect to $\mathbb{Z}^d$-translations.

\section{Introduction}
In this paper, we consider stochastic homogenization of the following linear elliptic equation 
\begin{equation}
\begin{dcases}
-\nab \cdot \left( \aa \left( \frac{x}{\eps} \right) \nab u_{\eps} \right)=0 &( \text{in } U) \text{,}  \\
u_{\eps}=f  &( \text{on } \partial U) \text{,} \label{int1}
\end{dcases}
\end{equation}
in a bounded open subset $U \subset \mathbb{R}^d$. The coefficient $\aa(\cdot) \in \mathbb{R}^{d  \times d} $ is a random field valued in the positive definite matrices and whose law under a probability measure $\prb$ has stationarity with respect to $\mathbb{Z}^d$-translations and ergodicity. In the situation where the coefficient $\aa(\cdot)$ is uniformly elliptic, Kozlov \cite{MR542557} and Papanicolaou and Varadhan \cite{MR712714} led to the qualitative stochastic homogenization result that the unique solution $u_{\eps}$ of \eqref{int1} converges, as $\eps \to +0$, to the solution $u$ of a constant coefficient equation
\begin{equation}
\begin{dcases}
-\nab \cdot \left( \bar{\aa} \nab u \right)=0 &( \text{in } U) \text{,}  \\
u=f  &( \text{on } \partial U) \text{.}
\end{dcases}
\end{equation}

Quantitative stochastic homogenization is a field that aims to obtain a rate of the convergence of $u_{\eps}$ to $u$ under the assumption of quantitative mixing conditions in ergodicity. The first quantitative result was obtained by Yurinski\u{i} \cite{MR1117252} by using probabilistic arguments. Quantitative theory of stochastic homogenization for elliptic equations was developed by the work of Gloria and Otto \cite{MR3415388}, Fischer and Neukamm \cite{MR4302762} and Gloria, Neukamm and Otto \cite{MR4377865}, who led to an optimal estimate under a situation that \textit{spectral gap inequality} or \textit{multiscale logarithmic Sobolev inequality} holds as a mixing condition. The core of their argument is to combine the spectral gap inequality with inequalities related to linear elliptic equations, such as the Meyers estimate and the Caccioppoli inequality, to derive a precise estimate for the moment of the gradient of the corrector which is a central object in the stochastic homogenization of linear elliptic equations. Also, a theory of quantitative stochastic homogenization and large-scale regularity under different assumptions such as \textit{finite-range of dependence} as mixing conditions was established by Armstrong and Smart \cite{MR3481355} and Armstrong, Kuusi and Mourrat \cite{MR3545509,MR3932093}. They succeeded in obtaining quantitative results by introducing another subadditive quantity corresponding to the subadditive energy quantity that played a central role in the qualitative theory by Maso and Modica \cite{MR870884,MR850613}. 

We are interested in extending the quantitative theory of stochastic homogenization to the case of unbounded and non-uniformly elliptic coefficients. In discrete setting, such attempts are preempted. For instance, Lamacz, Neukamm and Otto \cite{MR3418538} gave an error estimate of stochastic homogenization, and Andress and Neukamm \cite{MR3949105} gave a Berry-Esseen type estimate on random conductance model on the lattice $\mathbb{Z}^d$ satisfying the asuumptions of a spectral gap estimate and a moment condition. Also, in the continuous setting, Bella and Sch\"{a}ffner \cite{MR4201290} obtained the almost sure $L^{\infty}$-sublinearity of the corrector as a result of qualitative stochastic homogenization. On the other hand, to obtain quantitative stochastic homogenization, it was not known how much integrability of the ellipticity of the coefficient $\aa(\cdot)$ in the unit space of $\mathbb{R}^d$ is needed. Recently, Bella and Kniely \cite{bella2022regularity} generalized results of \cite{MR4377865} on the decay of the corrector gradient, growth of the corrector, and a quantitative two-scale expansion to the unbounded setting where averages of $|\aa|^p$ and $|\aa^{-1}|^q$ in unit ball possess stretched exponential moment bounds. However, they demanded that the set of coefficient fields is invariant under $\mathbb{R}^d$-transitions and a spectral gap condition by technical reasons. 

In this paper, we generalize the quantitative theory of \cite{MR3932093} to unbounded and non-uniformly elliptic case. An advantage of this approach is that estimates can be obtained without the assumptions of the spectral gap and other Poincar\'{e}-type inequalities. In this paper, we assume only  uniformly decreasing of maximal correlation $\rho$ as the mixing condition. Also, for the stationarity assumption, we assume stationarity with respect to $\mathbb{Z}^d$-transitions. We attempt to extend the subadditive argument by introducing the maximum and minimum values of the ellipticity of the coefficients for bounded Lipschitz domains as random variables. The difficulty of extending to the setting where the coefficients only satisfy local boundedness is that as the domain expands, the maximum values of $|\aa|$ and $|\aa^{-1}|$ in the domain increase. We show that the convergence of the subadditive quantities is rapid enough when the maximum values of $|\aa|$ and $|\aa^{-1}|$ in the unit domain have exponential integrability, and we obtain an estimate for the rate of the convergence of the solution of elliptic equations of the Dirichlet problem.

\subsection{Assumptions and Examples} \label{sec:assumption}
Let $\Om$ denote the set of measurable maps $\aa(\cdot)$ from $\mathbb{R}^d$ into the set of positive definite $d \times d$ matrices satisfying the following condition. 
\begin{itemize}
\item (Local boundedness). For every bounded domain $U\subset \mathbb{R}^d$,
\begin{equation}
\sup_{x \in U} \left| \aa(x) \right| + \sup_{x \in U} \left| \aa(x)^{-1} \right| < \infty. \label{as1}
\end{equation}
\end{itemize}
For a Borel subset $U \subseteq \mathbb{R}^d$, $\mathcal{F}_U$ denotes the $\sigma$-algebra generated by the following family of maps,
\begin{equation}
\left\{ \aa \mapsto \int_{\mathbb{R}^d} e_i \cdot \aa(x) e_j \varphi(x) dx : i,j \in \{1,2,\dots,d \},\varphi \in C_c^{\infty}(U) \right\}.
\end{equation}
$\mathcal{F}_{U}$ expresses the information of the coefficients $\aa(\cdot)$ in $U$. We write $\mathcal{F}:=\mathcal{F}_{\mathbb{R}^d}$. For each $y \in \mathbb{R}^d$, we let $T_y:\Om \to \Om$ be the shift operator such that
\begin{equation}
(T_y\aa)(x):=\aa(x+y).
\end{equation}
By using same notation, we let $T_y:\mathcal{F} \to \mathcal{F}$ be the map given by $T_{y}(A):=\{ T_y\aa:\aa \in A\}$. We let $d_{\infty}$ be the $l^{\infty}$-distance between two domains: for $U,V \subset \mathbb{R}^d$,
\begin{equation}
d_{\infty}(U,V):=\inf \left\{ \| u-v\|_{\infty}= \max_{1\leq i \leq d} |u_i-v_i | : u \in U,v \in V \right\}.
\end{equation}

We denote by $\prb$ a probability measure on the measurable space $(\Om, \mathcal{F})$. We write $\ep$ for the expectation with respect to $\prb$. For any $\sigma$-algebra $\mathcal{A} \subset \mathcal{F}$, denote $\mathcal{L}^2 \left( \mathcal{A} \right)$ the space of square-integrable, $\mathcal{A}$-measurable real-valued random variables. We assume that $( \Om, \mathcal{F}, \prb)$ satisfies the following properties corresponding to stationarity and ergodicity.
\begin{itemize}
\item (Stationarity with respect to $\mathbb{Z}^d$-translations). For all $z \in \mathbb{Z}^d$ and $A \in \mathcal{F}$,
\begin{equation}
\prb[A]=\prb[T_zA]. \label{as2}
\end{equation}
\item (Uniformly decreasing of maximal correlation $\rho$). Fix $r \in (0,1)$. For every pair of Borel subsets $U,V \subset \mathbb{R}^d$ with $d_{\infty}(U,V)\geq 1$, 
\begin{equation}
\rho \left( \mathcal{F}_{U}, \mathcal{F}_{V}\right) := \sup_{\substack{f \in \mathcal{L}^2 \left( \mathcal{F}_{U} \right) \\ g\in \mathcal{L}^2 \left( \mathcal{F}_{V} \right)} } \frac{\text{Cov}[f,g]}{\text{Var}[f]^{1/2}\text{Var}[g]^{1/2}} \leq r . \label{as3}
\end{equation}
\end{itemize}
\begin{rem}
\textup{It is not essential that the size of the shift in the stationarity assumption equals to the distance between the two sets in the ergodicity assumption. Indeed, the assumption of uniformly decreasing of maximal correlation $\rho$ can be rewritten as follows.}
\begin{itemize}
\item \textup{Fix $R<\infty$ and $r \in (0,1)$. For every pair of Borel subsets $U,V \subset \mathbb{R}^d$},
\begin{equation}
d_{\infty}(U,V) \geq R \quad \Rightarrow \quad \rho \left( \mathcal{F}_{U}, \mathcal{F}_{V}\right)\leq r . 
\end{equation}
\end{itemize}
\textup{If we replace the condition above by this condition, the constants $C$, $c$ in this paper will depend on $R$.}
\end{rem}
The maximal correlation $\rho(\mathcal{A},\mathcal{B})$ is the distance which measures the degree of dependence between the $\sigma$-algebras $\mathcal{A}$ and $\mathcal{B}$, and it was first studied in the papers \cite{MR3543,hirschfeld1935connection} in statistical contexts. In a random field, it is known that under certain conditions the $\rho$-mixing condition is equivalent to the $\alpha$-mixing condition, which is the weakest strong mixing condition, see \cite{MR1245294,MR2325296} for details. In the first quantitative result \cite{MR1117252}, quantitative estimates were obtained from the assumption of polynomial decay of the maximal correlation. In this paper, it is sufficient if $\rho$ decreases by a certain ratio at a certain distance. 

We introduce the random variables for spatial local boundedness as follows: for a bounded Lipschitz domain $U\subseteq \mathbb{R}^d$,
\begin{equation}
\begin{split}
\Lam(U)&:= \esssup\displaylimits_{x \in U} \sup_{\xi \in \mathbb{R}^d ,|\xi|=1} \xi \cdot \mathbf{a}(x) \xi ,\\
\lam(U)&:= \essinf\displaylimits_{x \in U} \inf_{\xi \in \mathbb{R}^d ,|\xi|=1} \xi \cdot \mathbf{a}(x) \xi .\label{def:lam}%
\end{split}
\end{equation}
$\Lam(U)$ and $\lam(U)$ are quantities representing the maximum and minimum values of the coefficients in the domain $U$. Note that local boundedness implies that, for every bounded domain $U$,
\begin{equation}
0<\lam(U)\leq \Lam(U)< \infty. \label{con:lam}
\end{equation}

An example that satisfies the assumptions of local boundedness \eqref{as1}, stationarity \eqref{as2}, and mixing condition \eqref{as3} is a random checkerboard model. We obtain the model by dividing $\mathbb{R}^d$ into cubes of size $1$ and assigning coefficients to be random variables that are independent and identically distributed on each cubes. Specifically, let $\{ b(z) \}_{z \in \mathbb{Z}^d}$ be independent and identically distributed random variables. We prepare a map $\mathbb{R} \ni r \mapsto \aa_r \in \mathbb{R}^{d \times d}$ from $\mathbb{R}$ to the set of positive definite matrices. Then, we can define a random field $\aa(\cdot): \mathbb{R}^d \to \mathbb{R}^{d \times d}$ satisfying \eqref{as1}, \eqref{as2} and \eqref{as3} by setting, for every $z \in \mathbb{Z}^d$ and $x \in z+\left[ -1/2,1/2 \right)^d$,
\begin{equation}
\aa(x):=\aa_{b(z)}.
\end{equation}
In the random checkerboard model, when at least one of $|\aa_{b(0)|}$ and $|\aa_{b(0)}^{-1}|$ is an unbounded ramdom variable, i.e., for every $C < \infty$,
\begin{equation}
\prb \, \left[\,|\aa_{b(0)}| + |\aa_{b(0)}^{-1}|  > C\, \right] >0,
\end{equation}
the model does not satisfy the uniform ellipticity of the coefficients.

\subsection{Main Results}

The main theorem of this paper is the quantitative stochastic homogenization of linear elliptic equations with unbounded and non-uniformly elliptic coefficients. We consider the situation where the maximum values of the operator norm of the coefficients and its inverse in the unit cube $\sq_0:=(-1/2,1/2)^d$ are unbounded random variables. In the following result, a homogenization rate is obtained when the random variables have exponential integrability. This is an estimate for the error in homogenization of the Dirichlet problem.

\begin{thm}\label{thm:king}
Let the assumptions in Section \ref{sec:assumption} hold. Let $\beta,\gamma \in (3,\infty]$ satisfy $0<\frac{1}{\beta}+\frac{1}{\gamma}<\frac{1}{3}$. Suppose that there exists $M < \infty$ satisfying 
\begin{equation}
\ep \left[ \exp \left( \Lam(\sq_0)^\beta \right) \right] + \ep \left[ \exp \left( \lam(\sq_0)^{-\gamma} \right) \right] \leq M. \label{thm:ass}
\end{equation} 
Then, for a bounded Lipschitz domain $U \subseteq \sq_0$, $\del >0$, $\alpha \in \left(\frac{1}{\beta}+\frac{1}{\gamma},\frac{1}{3}\right)$ and $p \in (0,4)$, there exist a symmetric matrix $\bar{\aa}$ and constants $c=c(d,r,\beta, \gamma ,M ,U,\del ,\alpha,p)>0$, $C=C(d,r,\beta, \gamma ,M ,U,\del ,\alpha,p) < \infty$ such that the following holds: For every $\eps \in (0,1]$, $f \in W^{1,2+\del}(U)$, and the unique solutions $u^\eps,u \in f+H_0^1(U)$ of the Dirichlet problems
\begin{alignat}{2}
-\nab \cdot \aa \left( \frac{\cdot}{\eps} \right) \nab u^\eps &=0 \quad(\text{in } U),  \qquad u^\eps&=f \quad( \text{on } \partial U), \\
-\nab \cdot \bar{\aa}  \nab u &=0 \quad(\text{in } U), \qquad u&=f \quad( \text{on } \partial U),
\end{alignat}
in the distribution sense, we have
\begin{equation}
\ep \left[ \left\|u-u^{\eps} \right\|_{L^2(U)}^p \right]^{\frac{1}{p}} \leq C \left\| \nab f \right\|_{L^{2+\del}(U)} \exp \left( -c\left( -\log \eps \right)^{1-3\alpha} \right). \label{thm:king2}
\end{equation}
\end{thm}

In this theorem, we assume only uniformly decreasing of maximal correlation $\rho$ as ergodicity. This condition is weaker than finite-range of dependence appeared in \cite{MR3545509,MR3932093,MR3481355}. We obtain quantitative result of unbounded and non-uniformly elliptic cases without multiscale logarithmic Sobolev inequality and spectral gap inequality, which are assumed in Bella and Kniely \cite{bella2022regularity}. Also, we assume only stationarity with respect to $\mathbb{Z}^d$-transitions. That is, it is not necessary that distributions of $\aa(\cdot)$ and $\aa(\cdot+x)$ coincide for every $x \in \mathbb{R}^d$.

The organization of the present paper is as follows. In Section $2.1$, we introduce subadditive quantities which are introduced in \cite{MR870884}, \cite{MR850613} and \cite{MR3481355}, and verify that the basic properties hold in our case. Then we characterize homogenized coefficients from the subadditivity of the quantities and estimate their sizes. Section $2.2$ is the core part of this paper. We obtain quantitative results on the convergence of the coefficients. We extend the previous iteration argument by introducing the concept of controlling the effects of unboundedness of the coefficients. In the first half of Section $3$, we show an estimates that leads from the convergence of the coefficients to the convergence of the solutions of the Dirichlet problem by a pointwise discussion of partial differential equations. Finally, we prove Theorem \ref{thm:king} in Section $3.3$.

\subsection{Notation}
We introduce some notations which are used throughout the paper. The set of nonnegative integers is denoted by $\mathbb{N}:=\{0,1,2,\dots\}$. The number of spatial dimensions is denoted by $d$. The canonical basis of $\mathbb{R}^d$ is written as $\{ e_1,e_2,\dots,e_d \}$, and the open ball of radius $r>0$ centered at $x \in \mathbb{R}^d$ is $B_r(x):=\left\{y\in \mathbb{R}^d: |y-x|<r \right\}$. For abbreviation, we write $B_1:=B_1(0)$. For $r>0$ and $U\subseteq \mathbb{R}^d$, we define $U_r:=\left\{ x \in U: \text{dist}(x,\partial U)>r \right\}$. For each $n \in \mathbb{N}$, we denote triadic cubes by
\begin{equation}
\sq_n:=\left( -\frac{1}{2}\cdot3^n,\frac{1}{2}\cdot3^n \right)^d.
\end{equation}
For $k \in \mathbb{N}$ and $f:U\to \mathbb{R}^d$, $\nab^k f$ is the tensor of $k$-times partial derivatives of $f$:
\begin{equation}
\nab^k f := \left( \partial_{x_{i_1}}\dots \partial_{x_{i_k}} f\right)_{i_1,i_2,\dots,i_k \in \{1,2,\dots,d\}}.
\end{equation}
For a measurable set $E\subseteq \mathbb{N}^d$, we denote the Lebesgue measure of $E$ by $|E|$. For simplicity of notation, we denote the integral of an integrable function $f:U\to \mathbb{R}$ by
\begin{equation}
\int_{U}f:=\int_{U}f(x)dx.
\end{equation}
When $0<|U|<\infty$ and $f \in L^1(U)$, we write
\begin{equation}
\fint_{U}f:=\frac{1}{|U|}\int_{U}f.
\end{equation} 
We often use the normalized $L^p$ norm: for every $p \in [1,\infty)$, $0<|U|<\infty$ and $f \in L^{p}(U)$, we set
\begin{equation}
\left\| f \right\|_{\ll^p(U)}:=\left( \fint_{U} \left| f\right|^p\right)^{\frac{1}{p}}=\left| U \right|^{-\frac{1}{p}} \left\| f \right\|_{L^p(U)}.
\end{equation}
We write $F\in L^p(U)$ if $F:U \to \mathbb{R}^m$ is a vector field such that $|F|\in L^p(U)$ and define 
\begin{equation}
\left\| F \right\|_{L^p(U)}:=\left\| \left| F\right| \right\|_{L^p(U)}\quad \text{and}\quad \left\| F \right\|_{\ll^p(U)}:=\left\| \left| F\right| \right\|_{\ll^p(U)}.
\end{equation}
For $p \in \mathbb{R}^d$, we denote the affine function with slope $p$ through the origin by $l_p(x):=p\cdot x$. We denote by $H^1(U)=W^{1,2}(U)$ the Sobolev space and denote by $H_0^1(U)$ the closure of $C_c^{\infty}(U)$ in $H^1(U)$. We define the normalized $H^1(U)$ norm of $u\in H^1(U)$  by
\begin{equation}
\| u\|_{ \underline{H}^1(U)}:= |U|^{-\frac{1}{d}}\|u\|_{\ll^2(U)}+\|\nab u \|_{\ll^2(U)}.
\end{equation}
Note that, for every $a>0$,
\begin{equation}
\left\| u\left( \frac{\cdot}{a}\right) \right\|_{ \underline{H}^1(aU)}=\frac{1}{a} \| u\|_{ \underline{H}^1(U)}.
\end{equation}
For every distribution $u$, the negative-order Sobolev norm is defined by
\begin{equation}
\left\| u \right\|_{H^{-1}(U)}:=\sup \left\{ \int_{U}uv : v \in H_0^1(U),\|v\|_{H^1(U)}\leq1\right\}.
\end{equation}
In this paper, we treat two different normalized $H^{-1}(U)$ norm. These are defined as follows:
\begin{align}
\left\| u \right\|_{\underline{H}^{-1}(U)}&:=\sup \left\{ \fint_{U}uv : v \in H_0^1(U),\|v\|_{\underline{H}^1(U)}\leq1\right\}, \\
\left\| u \right\|_{\hh^{-1}(U)}&:=\sup \left\{ \fint_{U}uv : v \in H^1(U),\|v\|_{\underline{H}^1(U)}\leq1\right\}.
\end{align}
Note that we denote by $\int_{U}uv$ the duality pairing between $u$ and $v$, and we understand it as $\fint_{U}uv=|U|^{-1}\int_{U}uv$ in the above definitions. Observe that
\begin{equation}
\left\| u\left( \frac{\cdot}{a}\right) \right\|_{ \underline{H}^{-1}(aU)}=a \| u\|_{ \underline{H}^{-1}(U)} \quad \text{and} \quad \left\| u\left( \frac{\cdot}{a}\right) \right\|_{ \hh^{-1}(aU)}=a \| u\|_{ \hh^{-1}(U)}.
\end{equation}
$H_{\text{loc}}^1(U)$ denotes the set of functions on $U$ which belong to $H^1(V)$ whenever $V$ is bounded and $\overline{V}\subset U$. We denote the set of weak solutions of the equation
\begin{equation}
-\nab \cdot \left( \aa \nab u\right)=0 \quad \text{in}\; U
\end{equation}
by $\mathcal{A}(U)\subseteq H_{\text{loc}}^1(U)$. Note that $\mathcal{A}(U)$ is a set of weak solutions to $-\nab \cdot \left( \aa \nab u\right)=0$ in the distribution sense. It is equivalent to the condition
\begin{equation}
\int_{U} \nab \phi \cdot \aa \nab u=0 \quad \text{for every }\phi \in H_0^1(U).
\end{equation}
We extend the notations $\text{Var}[X]$ and $\text{Cov}[X,Y]$ to ramdom vectors $X,Y:\Om \to \mathbb{R}^d$ by setting 
\begin{equation}
\text{Var}[X]:=\ep \left[ \left| X-\ep[X]\right|^2\right] \quad\text{and} \quad \text{Cov}[X,Y]:= \ep \left[ \left( X-\ep[X]\right)\cdot\left( Y-\ep[Y]\right)\right].
\end{equation}
The identity matrix is denoted by \textup{\textsf{Id}}. For $A \in \mathbb{R}^{d \times d}$, let $|A|$ be the operator norm of $A$, defined by
\begin{equation}
|A|:=\sup_{\xi \in \mathbb{R}^d ,|\xi|=1} |A\xi|.
\end{equation}
We define a partial ordered relation $\leq$ in a set of positive definite $d \times d$ matrices as follows:
\begin{equation}
A\leq B \quad \Leftrightarrow \quad p\cdot(B-A)p \geq 0 \; \text{ for every }p\in\mathbb{R}^d.
\end{equation}
Note that the ordered relation has the following property for inverses: let $A,B \in \mathbb{R}^{d \times d}$ be positive definite, then 
\begin{equation}
A^{-1}\leq B^{-1} \quad \Leftrightarrow \quad A\geq B.
\end{equation}
For the proof we refer the reader to \cite[Theorem 2.2.5.]{murphy1990c}. Throughout the paper, $c$ and $C$ denote positive constants which may vary without notice from line to line and also between multiple expressions in the same line.

\section{Convergence of Subadditive Quantities}

\subsection{The subadditive Quantities $\mu$, $\mu_*$ and $J$}
We recall subadditive quantities from \cite{MR3932093} which play key roles in this paper. For a bounded Lipschitz domain $U\subset \mathbb{R}^d$ and $p \in \mathbb{R}^d$, we denote $l_p(x):=p\cdot x$ and define
\begin{equation}
\mu(U,p):=\inf_{v\in l_p+H_0^1(U)} \fint_{U} \frac{1}{2} \nab v \cdot \aa \nab v =\inf_{w\in H_0^1(U)} \fint_{U} \frac{1}{2}(p+\nab w)\cdot \aa(p+\nab w). \label{def:mu}
\end{equation}
We denote by $v(\cdot,U,p)$ the unique $v\in l_p+H_0^1(U)$ minimizing $\fint_{U} \frac{1}{2} \nab v \cdot \aa \nab v$.
The existence and uniqueness of $v(\cdot,U,p)$ or maximizer of \eqref{def:mu*} below derives from \eqref{con:lam}. We mention that $v(\cdot,U,p)$ is characterized as the unique weak solution of the Dirichlet problem
\begin{equation}
-\nab \cdot \aa  \nab v =0 \quad(\text{in } U), \qquad v=l_p \quad( \text{on } \partial U),  \label{def:weak}
\end{equation}
in the distribution sense. In other words, $v(\cdot,U,p) \in l_p+H_0^1(U)$ satisfies, for every $w \in H_0^1(U)$,
\begin{equation}
\fint_{U} \nab w \cdot \aa \nab v(\cdot,U,p)=0.
\end{equation}
For a bounded Lipschitz domain $U\subseteq \mathbb{R}^d$ and $p \in \mathbb{R}^d$, we also define the dual subadditive quantity $\mu_*(U,q)$ by
\begin{equation}
\mu_*(U,q):=\sup_{u \in \mathcal{A}(U)} \left( \fint_{U} -\frac{1}{2} \nab u \cdot \aa \nab u+q\cdot \nab u \right). \label{def:mu*}%
\end{equation}
Recall that $\mathcal{A}(U)$ is linear space of weak solutions to $-\nab\cdot\left( \aa \nab u\right)=0$. We define $J(U,p,q)$ by
\begin{equation}
J(U,p,q):=\mu(U,p)+\mu_*(U,q)-p\cdot q .\label{def:J}%
\end{equation}
Our goal is to obtain convergence of $J(\sq_n,p,\bar{\aa}p)$ to 0. First, we refer to another characterization of $J$ in the following lemma.


\begin{lem} \label{lemlem:J}
For every bounded Lipschitz domain $U \in \mathbb{R}^d$ and $p,q \in \mathbb{R}^d$,
\begin{equation}
J(U,p,q)=\sup_{w \in \mathcal{A}(U)} \fint_{U} \left( -\frac{1}{2}\nab w \cdot \aa \nab w-p \cdot \aa \nab w +q\cdot \nab w \right) . \label{lem:J}%
\end{equation}
Also, the maximizer $v(\cdot,U,p,q)$ of \eqref{lem:J} is the difference of the maximizer of $\mu_*(U,q)$ in \eqref{def:mu*} and $v(\cdot,U,p)$.
\end{lem}

\begin{proof}
See \cite[Lemma 2.1]{MR3932093}.
\end{proof}

It is easily seen that $v(\cdot,U,p,q)$ is unique up to additive constants. For convenience, we choose this additive constant so that $\fint_{U}v(\cdot,U,0,q)=0$ and $v(\cdot,U,p,0)=-v(\cdot,U,p)$.

By testing the right-hand side of \eqref{lem:J} with 0, we see that for every $p,q\in \mathbb{R}^d$,
\begin{equation}
\mu(U,p)+\mu_*(U,q)-p \cdot q=J(U,p,q) \geq 0. \label{ieJ} %
\end{equation}


The following lemma is some basic properties of $J(U,p,q)$ and $v(\cdot,U,p,q)$.

\begin{lem}[Properties of $J$] Fix a bounded Lipscitz domain $U \subseteq \mathbb{R}^d$. $J(U,p,q)$ and $v(\cdot,U,p,q)$ satisfy the following properties:

\begin{itemize}
\item \textup{(Representation as quadratic form)} The mapping $(p,q)\mapsto J(U,p,q)$ is a quadratic form and there exist symmetric matrices $\aa(U)$ and $\aa_*(U)$ such that
\begin{equation}
\lam(U)\textup{\textsf{Id}} \leq \aa_*(U) \leq \aa(U) \leq \Lam(U)\textup{\textsf{Id}} \label{lem:rq1} %
\end{equation}
and
\begin{equation}
J(U,p,q)=\frac{1}{2}p\cdot \aa(U)p+\frac{1}{2}q\cdot \aa_*^{-1}(U)q -p \cdot q \label{lem:rq2}%
\end{equation}
Moreover, $\aa(U)$ and $\aa_*(U)$ are characterized by the following relations, for every $p,q\in\mathbb{R}^d$:
\begin{align}
\aa(U)p=-&\fint_{U}\aa \nab v(\cdot,U,p,0), \label{lem:rq3} \\ %
\aa_*^{-1}(U)q=&\fint_{U} \nab v(\cdot,U,0,q). \label{lem:rq4} %
\end{align}
\item \textup{(Subadditivity)} Let $U_1,U_2,\dots,U_N\subseteq U$ be bounded Lipschitz domains that are a partition of $U$, in the sense that 
\begin{equation}
U_i \cap U_j=\emptyset\;\; \text{if $i \neq j$,} \quad \text{and} \quad \left| U\setminus \bigcup_{i=1}^{N} U_i \right|=0. \label{lem:sb1} %
\end{equation} 
Then, for every $p,q\in \mathbb{R}^d$,
\begin{equation}
J(U,p,q) \leq \sum_{i=1}^{N} \frac{\left|U_i\right|}{\left|U\right|}J(U_i,p,q). \label{lem:sb2}%
\end{equation} 
\item \textup{(First variation for $J$)} For $p,q \in \mathbb{R}^d$, $v(\cdot,U,p,q)$ is characterized as the unique function in $\mathcal{A}(U)$ which satisfies, for every $w \in \mathcal{A}(U)$,
\begin{equation}
\fint_{U} \nab w \cdot \aa \nab v(\cdot,U,p,q)=\fint_{U} \left( -p \cdot \aa \nab w +q \cdot \nab w \right). \label{lem:fv} %
\end{equation}
\item \textup{(Quadratic response)} If $U_1,U_2,\dots,U_N\subseteq U$ be bounded Lipschitz domains that is a partition of $U$ in the sense of \eqref{lem:sb1}, then
\begin{align}
&\sum_{i=1}^{N} \frac{\left|U_i\right|}{\left|U\right|}\fint_{U_i}\frac{1}{2}\left(\nab v (\cdot,U,p,q)-\nab v (\cdot,U_i,p,q) \right) \cdot \aa \left(\nab v (\cdot,U,p,q)-\nab v (\cdot,U_i,p,q) \right) \\
&=  \sum_{i=1}^{N} \frac{\left|U_i\right|}{\left|U\right|} \left(J(U_i,p,q)-J(U,p,q)\right). \label{lem:qr2} %
\end{align}
\end{itemize}
\end{lem}


\begin{proof}
From \eqref{def:lam} and \eqref{con:lam}, applying similar arguments to the proof of \cite[Lemmas 2.2 and 2.12]{MR3932093}, we obtain \eqref{lem:sb2}, \eqref{lem:fv} and \eqref{lem:qr2}. To obtain the existence of $\aa_*(U)$ and the bounds of $\aa(U)$ and $\aa_*(U)$, we will prove representation as quadratic form. By \eqref{lem:fv}, we have that
\begin{equation}
(p,q) \mapsto v(\cdot,U,p,q) \text{ is linear.} \label{vli} %
\end{equation}
We observe that we have the following identity:
\begin{equation}
J(U,p,q)=\fint_{U} \frac{1}{2} \nab v(\cdot,U,p,q) \cdot \aa \nab v(\cdot,U,p,q). \label{lem:Jq} %
\end{equation}
Indeed, this is immediate from \eqref{lem:J} and \eqref{lem:fv}. This together with \eqref{vli} implies 
\begin{equation}
(p,q) \mapsto J(U,p,q) \text{ is quadratic.} \label{lem:Jqd} %
\end{equation}
From \eqref{def:mu} and \eqref{def:mu*}, it follows that $\mu(U,0)=0$ and $\mu_*(U,0)=0$. In view of the formula \eqref{def:J} and \eqref{lem:Jq}, this implies in particular that 
\begin{equation}
p \mapsto \mu(U,p) \text{ and } q \mapsto \mu_*(U,q) \text{ are quadratic.} \label{lem:muqd} %
\end{equation}
We define $\aa(U),\hat{\aa}_*(U)$ to be the symmetric matrices such that for every $p,q\in \mathbb{R}^d$,
\begin{align}
\mu(U,p)&=\frac{1}{2}p\cdot \aa(U)p, \label{def:mu2}\\  %
\mu_*(U,q)&=\frac{1}{2}q \cdot \hat{\aa}_*(U)q. \label{def:mu*2}  %
\end{align} 
From \eqref{lem:Jq} and \eqref{def:J}, we have
\begin{align}
p \cdot \aa(U) p &=\fint_{U} \nab v(\cdot,U,p,0) \cdot \aa \nab v(\cdot,U,p,0), \\
q \cdot \hat{\aa}_*(U)q &=\fint_{U} \nab v(\cdot,U,0,q) \cdot \aa \nab v(\cdot,U,0,q).
\end{align}
This and \eqref{vli} imply that, for every $p,q \in \mathbb{R}^d$,
\begin{align}
q \cdot \aa(U) p &=\fint_{U} \nab v(\cdot,U,p,0) \cdot \aa \nab v(\cdot,U,q,0), \\
p \cdot \hat{\aa}_*(U)q&=\fint_{U} \nab v(\cdot,U,0,q) \cdot \aa \nab v(\cdot,U,0,p).
\end{align}
By \eqref{lem:fv}, we have
\begin{align}
q \cdot \aa(U) p &= -\fint_{U} q \cdot \aa \nab v(\cdot,U,p,0), \\
p \cdot \hat{\aa}_*(U)q&=\fint_{U} p \cdot \nab v (\cdot,U,0,q). 
\end{align}
These imply that
\begin{align}
\aa(U)p&=-\fint_{U} \nab v(\cdot,U,p,0)  \label{pr:aaup} \\ 
\hat{\aa}_*(U)q&=\fint_{U} \nab v(\cdot,U,0,q). \label{pr:hatq}
\end{align}
We prove the bounds of $\aa(U)$ and $\hat{\aa}_*(U)$, and the existence of the inverse matrix of $\hat{\aa}_*(U)$. The upper bound of $\aa(U)$ is immediate from testing the definition of $\mu(U,p)$ with $l_p$:
\begin{equation}
p \cdot \aa(U)p=2\mu(U,p)\leq \fint_{U} \nab l_p \cdot \aa \nab l_p =\fint_{U} p\cdot \aa p \leq \Lam(U)|p|^2. \label{mub} %
\end{equation}
This implies that $\aa(U) \leq \Lam(U)\textup{\textsf{Id}}$. The lower bound of $\aa(U)$ comes from Jensen's inequality: for every $w \in H_0^1(U)$,
\begin{equation}
\fint_{U}\frac{1}{2}(p +\nab w)\cdot \aa (p+\nab w) \geq \fint_{U} \frac{\lam(U)}{2}|p+\nab w|^2 \geq \frac{\lam(U)}{2}\left| p+\fint_{U}\nab w \right|^2 =\frac{\lam(U)}{2}|p|^2.
\end{equation}
Taking the infimum over $w\in H_0^1(U)$ and \eqref{def:mu} yields the lower bound $\lam(U)\textup{\textsf{Id}} \leq \aa(U)$. By the bound of $\aa(U)$, we have the existence of $\aa^{-1}(U)$ which is positive definite. From \eqref{ieJ}, \eqref{def:mu2} and \eqref{def:mu*2}, we have
\begin{equation}
\frac{1}{2}p \cdot \aa(U)p+ \frac{1}{2} q \cdot \hat{\aa}_*(U)q \geq p\cdot q
\end{equation}
By substituting $p=\aa^{-1}(U)q$ in this inequality, it follows that $\aa^{-1}(U) \leq \hat{\aa}_*(U)$ and that $\hat{\aa}_*(U)$ is also positive definite. We define $\aa_*(U)$ by the positive definite matrix such that
\begin{equation}
\aa_*(U):=\hat{\aa}_*^{-1}(U). \label{def:mu*3}
\end{equation}
From \eqref{def:J}, \eqref{pr:aaup}, \eqref{pr:hatq}, and \eqref{def:mu*3}, we have \eqref{lem:rq2}, \eqref{lem:rq3} and \eqref{lem:rq4}. Since $\aa^{-1}(U)$ and $\aa^{-1}_*(U)$ are positive, the inequality $\aa^{-1}(U) \leq \aa^{-1}_*(U)$ leads to the bound $\aa_*(U) \leq \aa(U)$. To obtain the lower bound for $\aa_*(U)$, we first use \eqref{lem:rq4} to write
\begin{equation}
\frac{1}{2}q \cdot \aa_*^{-1}(U)q= \frac{1}{2} \fint_{U} q \cdot \nab v(\cdot,U,0,q).
\end{equation}
By Young's inequality,
\begin{equation}
\fint_{U}q \cdot \nab v (\cdot,U,0,q) \leq \fint_{U} \left(\frac{1}{2}q \cdot \aa^{-1}q+ \frac{1}{2} \nab v(\cdot,U,0,q) \cdot \aa \nab v(\cdot,U,0,q) \right),
\end{equation}
and thus, by \eqref{lem:rq2} and \eqref{lem:Jq},
\begin{equation}
q \cdot \aa_*^{-1}(U)q \leq \fint_{U} q \cdot \aa^{-1}q \leq \frac{1}{\lam(U)}|q|^2,
\end{equation}
which gives the desired lower bound $\lam(U)\textup{\textsf{Id}} \leq \aa_*(U)$. This completes the proof of the lemma.
\end{proof}

From \eqref{def:mu*2} and \eqref{def:mu*3}, we have 
\begin{equation}
\mu_*(U,q)= \frac{1}{2}q \cdot \aa_*^{-1}(U)q. \label{def:mu*4}
\end{equation}
We also note that \eqref{lem:rq1}, \eqref{def:mu2} and \eqref{def:mu*4} lead to
\begin{align}
\frac{1}{2}\lam(U)|p|^2 &\leq \mu(U,p) \leq \frac{1}{2}\Lam(U)|p|^2, \label{ieq:mu} \\
\frac{1}{2\Lam(U)} |q|^2 &\leq \mu_*(U,q) \leq \frac{1}{2\lam(U)} |q|^2. \label{ieq:mu*}
\end{align}
From this inequality and \eqref{def:J}, we obtain
\begin{equation}
J(U,p,q) \leq \frac{1}{2}\Lam(U)|p|^2 + \frac{1}{2\lam(U)} |q|^2 + |p||q|. \label{ieq:Jpq}
\end{equation}
This inequality and \eqref{lem:Jq} imply that
\begin{align}
\fint_{U} \left| \nab v(\cdot,U,p,q) \right|^2 &\leq \frac{1}{\lam(U)} \fint_{U} \nab v(\cdot,U,p,q) \cdot \aa \nab v(\cdot,U,p,q) \\
&=\frac{2}{\lam(U)} J(U,p,q) \\
&\leq \frac{\Lam(U)}{\lam(U)}|p|^2 + \frac{1}{\lam(U)^2} |q|^2 + \frac{1}{\lam(U)} |p||q|. \label{ieq:v2}
\end{align} 

The following lemma is the basis on which estimating quantities about $J$ leads to the convergence of $\aa(\sq_n)$ to $\bar{\aa}$.

\begin{lem}\label{lemlem:|J}
There exists a constant $C< \infty$ such that, for every symmetric matrix $\tilde{\aa} \in \mathbb{R}^{d\times d}$ and every bounded Lipschitz domain $U \subseteq \mathbb{R}^d$, we have
\begin{equation}
|\aa(U)-\tilde{\aa}| \leq C \Lam(U)^{\frac{1}{2}}\sup_{p \in B_1} \left( J(U,p,\tilde{\aa}p) \right)^{\frac{1}{2}}. \label{lem:aJ} %
\end{equation}
\end{lem}
\begin{proof}
By \eqref{lem:rq1} and \eqref{lem:rq2}, for every $p,q \in \mathbb{R}^d$, we have that 
\begin{align}
|\aa(U)p-q|^2 \leq & \Lam(U) \left( (\aa(U)p-q)\cdot \aa^{-1}(U) (\aa(U)p-q) \right) \\
=&\Lam(U) \left( p \cdot \aa(U)p +q \cdot \aa^{-1}(U)q -2p\cdot q \right) \\
\leq&\Lam(U) \left( p \cdot \aa(U)p +q \cdot \aa_*^{-1}(U)q -2p\cdot q \right) \\
=&2\Lam(U)J(U,p,q).
\end{align}
Choosing $q=\tilde{\aa}p$ and taking the supremum over $p \in B_1$, we obtain \eqref{lem:aJ}.
\end{proof}

Let us define homogenized coefficients $\bar{\aa}$. Subadditivity and stationarity lead to the monotonicity of $\ep \left[ \mu( \square_m,p) \right]$: for every $m \in \mathbb{N}$ and $p \in \mathbb{R}^d$,
\begin{equation}
\ep \left[ \mu( \square_{m+1},p) \right] \leq \ep \left[ \mu( \square_m,p) \right]. \label{momu} %
\end{equation}
To see this, we first apply \eqref{lem:sb2} with respect to the partition $\{z+\square_m:z\in \{ -3^m,0,3^m\}^d\}$ of $\square_{m+1}$ into its $3^d$ largest triadic subcubes, to get
\begin{equation}
\mu(\square_{m+1},p)\leq 3^{-d} \sum_{z \in \{ -3^m,0,3^m\}^d} \mu(z +\square_{m},p).
\end{equation} 
From stationarity, for every $z \in \mathbb{Z}^d$, the law of $\mu(z+\square_m,p)$ is same as the law of $\mu(\square_m,p)$. Taking the expectation of the previous display gives \eqref{momu}. By similar arguments, we have: for every $p,q \in\mathbb{R}^d$
\begin{align}
\ep \left[ \mu_*( \square_{m+1},q) \right] &\leq \ep \left[ \mu_*( \square_m,q) \right], \label{momu*} \\
\ep \left[ J(\square_{m+1},p,q) \right]  & \leq \ep \left[ J(\square_m,p,q) \right]. \label{moJ} 
\end{align}
Using the positivity of $\aa(U)$ and \eqref{def:mu2} we see that, for each $p \in \mathbb{R}^d$, the sequence $\{\ep[\mu(\square_m,p)]\}_{m \in \mathbb{N}}$ is bounded below and nonincreasing. It therefore has a limit, which we denote by
\begin{equation}
\bar{\mu}(p):= \lim_{m \to \infty} \ep[\mu(\square_m,p)] = \inf_{m \in \mathbb{N}} \ep[\mu(\square_m,p)]. \label{def:bmu} %
\end{equation}
From \eqref{lem:muqd} it follows that $p \mapsto \ep[\mu(U,p)]$ is also quadratic, and hence
\begin{equation}
p \mapsto \bar{\mu}(p) \text{ is quadratic}. \label{bmuqd} %
\end{equation}
By these facts, the homogenized cofficient $\bar{\aa} \in \mathbb{R}^{d \times d}$ can be defined as the symmetric matrix satisfying
\begin{equation}
\bar{\mu}(p)=\frac{1}{2}p\cdot \bar{\aa}p \quad (p \in \mathbb{R}^d).  \label{def:ba} %
\end{equation}

We give an estimate of the homogenized coefficients from local informations.
\begin{prop} \label{propprop:baa}
If $\ep[\Lam(\sq_0)]$ and $\ep[\lam(\square_0)^{-1}]$ are finite, then homogenized coefficients $\bar{\aa}$ is positive definite, and satisfies
\begin{equation}
\ep[\lam(\square_0)^{-1}]^{-1} \textup{\textsf{Id}} \leq \bar{\aa} \leq \ep[\Lam(\square_0)] \textup{\textsf{Id}}. \label{prop:baa} %
\end{equation}
\end{prop}
\begin{proof}
The upper bound is immmediate from \eqref{def:bmu} and \eqref{mub}: for every $p \in\mathbb{R}^d$
\begin{equation}
p \cdot \bar{\aa}p = 2\bar{\mu}(p) \leq \ep[2\mu(\square_0,p)] \leq \ep[\Lam(\square_0)]|p|^2.
\end{equation}

The lower bound comes from the \eqref{ieJ}, \eqref{momu*} and \eqref{ieq:mu*}: for every $m \in \mathbb{N}$ and every $p,q \in \mathbb{R}^d$
\begin{align}
2\ep[\mu(\square_m,p)] &\geq 2p\cdot q - 2\ep[\mu_*(\square_m,q)]  \\
&\geq 2p\cdot q -2\ep[\mu_*(\square_0,q)] \\
&\geq 2p\cdot q - \ep[\lam(\square_0)^{-1}] |q|^2.
\end{align}
Choosing $q=\ep[\lam(\square_0)^{-1}]^{-1}p$ and taking the infimum over $m \in \mathbb{N}$, we have
\begin{equation}
p \cdot \bar{\aa}p \geq \ep[\lam(\square_0)^{-1}]^{-1} |p|^2.
\end{equation}
Thus, \eqref{prop:baa} is proved.
\end{proof}

\begin{rem} \textup{The upper bound of $\mu(U,p)$ in \eqref{mub} is improved as}
\begin{equation}
\mu(U,p) \leq \fint_{U}\frac{1}{2}|\aa||p|^2.
\end{equation}
\textup{By the previous formula and a similar argument, we get a more precise estimate of $\bar{\aa}$:}
\begin{equation}
\ep\left[ \fint_{\square_0} \left| \aa^{-1} \right|\right]^{-1} \textup{\textsf{Id}} \leq \bar{\aa} \leq \ep \left[ \fint_{\square_0} \left| \aa \right| \right] \textup{\textsf{Id}}.
\end{equation}
\end{rem} 


\subsection{Quantitative Convergence of Subadditive Quantity} \label{sec:QCSQ} 
 
The purpose of this section is to see the effect of the unboundedness and the non-uniformly ellipticity of the coefficients on the estimate of the rate of convergence of $\aa(\sq_n)$ to $\bar{\aa}$. We introduce the following concepts to control the influence of events with large or small coefficients. 
\begin{dfn}
Fix $L<\infty$. We call a pair of two positive real sequences $\left( \{\del_n\},\{M_n\} \right)$ is \textit{suppressive} if $\{\del_n\}$ is decreasing, $\{M_n\}$ is increasing, $\del_0=M_0=1$ and, for every $n \in \mathbb{N}$,
\begin{equation}
\ep \left[ \lam\left(\square_n\right)^{-3}+\Lam\left(\sq_n\right)^3:\left\{ \lam\left(\sq_n\right)\leq \del_n\right\} \cup \left\{ \Lam\left(\sq_n\right)\geq M_n\right\} \right]\leq Le^{-n}. \label{def:sup}
\end{equation}
\end{dfn}
Roughly speaking, if the coefficients in the unit cube have exponential integrability, we can take polynomials as suppressive sequences. We give an example of suppressive sequences. Fix $\beta,\gamma>0$ and suppose that there exists $M< \infty$ satisfying
\begin{equation}
\ep \left[ \exp \left( \Lam(\sq_0)^\beta \right) \right] + \ep \left[ \exp \left( \lam(\sq_0)^{-\gamma} \right) \right] \leq M.
\end{equation}
Then, for $\beta'>\frac{1}{\beta}$ and $\gamma'>\frac{1}{\gamma}$, there exists $L(\beta,\gamma,M,\beta',\gamma')< \infty$ such that 
\begin{equation}
\del_n:=(n+1)^{-\gamma'} \quad \text{and} \quad M_n:=(n+1)^{\beta'}
\end{equation}
are suppressive. The proof of this fact is in Section \ref{proofking}. 

We note that \eqref{def:sup} gives a constant $C=C(L)<\infty$ such that, for every $p,q \geq 0$ satisfying $p+q \leq 3$ and $n \in \mathbb{N}$,
\begin{equation}
\ep \left[ \Lam(\sq_n)^p\lam(\sq_n)^{-q}: \left\{ \lam\left(\sq_n\right)\leq \del_n\right\} \cup \left\{ \Lam\left(\sq_n\right)\geq M_n\right\}\right] \leq Ce^{-n}. \label{ieq:sup}
\end{equation}
Indeed, the fact that $\lam(\sq_n)\leq \Lam(\sq_n)$ gives, for every $A \in \mathcal{F}$,
\begin{align}
\ep \left[ \Lam(\sq_n)^p\lam(\sq_n)^{-q}: A \right]  &\leq \ep \left[ \left(\Lam(\sq_n)\lam(\sq_n)^{-1}\right)^{\frac{3-p-q}{2}} \Lam(\sq_n)^p\lam(\sq_n)^{-q}:A\right] \\
&= \ep \left[ \Lam(\sq_n)^{\frac{3+p-q}{2}}\lam(\sq_n)^{-\left(\frac{3-p+q}{2}\right)}:A\right] \\
&\leq \ep \left[ \max\{\Lam(\sq_n), \lam(\sq_n)^{-1}\}^3:A\right] \\
&\leq \ep \left[ \lam\left(\square_n\right)^{-3}+\Lam\left(\sq_n\right)^3:A \right]. \label{ieq:Lplq}
\end{align}
Thus, we have
\begin{align}
&\ep \left[ \Lam(\sq_n)^p\lam(\sq_n)^{-q}: \left\{ \lam\left(\sq_n\right)\leq \del_n\right\} \cup \left\{ \Lam\left(\sq_n\right)\geq M_n\right\} \right] \\
\leq \,& \ep \left[ \Lam(\sq_n)^3+\lam(\sq_n)^{-3}: \left\{ \lam\left(\sq_n\right)\leq \del_n\right\} \cup \left\{ \Lam\left(\sq_n\right)\geq M_n\right\} \right] \leq Le^{-n}.
\end{align}
We also note that \eqref{def:sup} leads to the existence of a constant $C=C(L)<\infty$ such that 
\begin{equation}
\ep [\Lam(\sq_0)] +\ep [ \lam(\sq_0)^{-1} ] \leq C. \label{supep0}
\end{equation}
Indeed, \eqref{ieq:Lplq} and \eqref{def:sup} show that, for every $p,q \geq 0$ satisfying $p+q \leq 3$,
\begin{equation}
\ep \left[ \Lam(\sq_0)^p\lam(\sq_0)^{-q}\right] \leq \ep \left[ \lam\left(\square_0\right)^{-3}+\Lam\left(\sq_0\right)^3 \right] \leq L+2. \label{ieq:Lplq2}
\end{equation}
Hence, this inequality yields \eqref{supep0}.

The main result of this section is the following theorem, which gives the quantitative convergence of $\aa(\square_n)$ to $\bar{\aa}$.
\begin{thm}\label{thm:coa}
Fix $\alpha \in (0,\frac{1}{3})$. Suppose that there exist suppressive sequences $\left( \{\del_n\},\{M_n\} \right)$ such that, for every $n \in \mathbb{N}$,
\begin{equation}
\frac{M_n}{\del_n} \leq (n+1)^{\alpha}.
\end{equation}
Then, there exist constants $c=c(d,r,\alpha,L)>0$ and $C=C(d,r,\alpha,L)<\infty$ such that, for every $n \in \mathbb{N}$,
\begin{equation} 
\ep\left[  \left| \bar{\aa}-\aa(\square_n) \right|^2  \right] \leq C \exp\left(-cn^{1-3\alpha}\right). \label{thm:coaa}
\end{equation}
\end{thm}

To prove this theorom, we will estimate $\ep\left[ J\left(\sq_n,p,q\right) \right]$ with 
\begin{equation}
\begin{split}
\tau_n:&=\sup_{p,q\in B_1} \left( \ep \left[ J\left( \sq_n,p,q \right) \right] -\ep \left[ J\left( \sq_{n+1},p,q \right) \right] \right) \\
&=\sup_{p \in B_1} \left( \ep \left[ \mu \left(\sq_n,p \right) \right]-\ep \left[ \mu \left(\sq_{n+1},p \right) \right] \right) + \sup_{q \in B_1} \left( \ep \left[ \mu_* \left(\sq_n,q \right) \right]-\ep \left[ \mu_* \left(\sq_{n+1},q \right) \right] \right). \label{def:tau}
\end{split}
\end{equation}
and use iteration. 
We first estimate the variance of the spatial average of the gradient of the maximizer of $J$.

\begin{lem} \label{lemlem:varsa} %
Let $\left( \{\del_n\},\{M_n\} \right)$ be suppressive. Then, there exist $\ka=\ka(d,r)>0$ and $C=C(d,r,L) < \infty$ such that, for every $p,q \in B_1$ and $n \in \mathbb{N}$,
\begin{equation}
\textup{Var}\left[ \fint_{\sq_n} \nab v(\cdot,\sq_n,p,q) \right] \leq Ce^{-\ka n}+\frac{C}{\del_{n}}\sum_{m=0}^{n-1} e^{-\ka (n-m-1)} \tau_m. \label{lem:varsa} %
\end{equation}
\end{lem}

\begin{proof}
\textit{Step1.} Fix $n \in \mathbb{N}$, $p,q \in B_1$ and set $v:=v(\cdot,\sq_{n+1},p,q)$, $v_z:=v(\cdot,z+\sq_n,p,q)$. We begin by proving that there exist $\ka'>0$ and $C=C(d,L)<\infty$ such that 
\begin{equation}
\text{Var} \left[\fint_{\sq_{n+1}} \nab v \right]^{1/2} \leq \text{Var}\left[ 3^{-d} \sum_{z \in 3^n\mathbb{Z}^d\cap \sq_{n+1}}\fint_{z+\sq_{n}}\nab v_z \right]^{1/2}+\frac{C}{\del_{n+1}^{1/2}}\tau_{n}^{1/2}+Ce^{-\ka' n}. \label{pr:var} %
\end{equation}
Using Cauchy-Schwarz inequality and Jensen's inequality, we get 
\begin{align}
&\text{Var}\left[ \fint_{\sq_{n+1}} \nab v \right]^{1/2} \\ 
\leq& \text{Var} \left[ 3^{-d} \sum_{z \in 3^n\mathbb{Z}^d\cap \sq_{n+1}} \fint_{z+\sq_n} \left( \nab v -\nab v_z \right) \right]^{1/2} +\text{Var} \left[ 3^{-d} \sum_{z \in 3^n\mathbb{Z}^d\cap \sq_{n+1}} \fint_{z+\sq_n} \nab v_z \right]^{1/2} \\
\leq& \ep \left[ 3^{-2d} \left| \sum_{z \in 3^n\mathbb{Z}^d\cap \sq_{n+1}} \fint_{z+\sq_n} \left( \nab v -\nab v_z \right) \right|^2 \right]^{1/2} +\text{Var} \left[ 3^{-d} \sum_{z \in 3^n\mathbb{Z}^d\cap \sq_{n+1}} \fint_{z+\sq_n} \nab v_z \right]^{1/2} \\
\leq& \ep \left[ 3^{-d} \sum_{z \in 3^n\mathbb{Z}^d\cap \sq_{n+1}}  \fint_{z+\sq_n} \left| \nab v -\nab v_z \right|^2 \right]^{1/2} +\text{Var} \left[ 3^{-d} \sum_{z \in 3^n\mathbb{Z}^d\cap \sq_{n+1}} \fint_{z+\sq_n} \nab v_z \right]^{1/2}. 
\end{align}
By \eqref{def:lam} and \eqref{lem:qr2}, we see that
\begin{align}
&\sum_{z \in 3^n\mathbb{Z}^d\cap \sq_{n+1}} \fint_{z+\sq_n} \left| \nab v -\nab v_z \right|^2 \\
&\leq \frac{1}{\lam(\sq_{n+1})} \sum_{z \in 3^n\mathbb{Z}^d\cap \sq_{n+1}} \fint_{z+\sq_n} (\nab v- \nab v_z)\cdot \aa (\nab v -\nab v_z) \\
&=\frac{2}{\lam(\sq_{n+1})}\sum_{z \in 3^n\mathbb{Z}^d\cap \sq_{n+1}} \left(J \left(z+\sq_n,p,q\right)-J\left(\sq_{n+1},p,q\right) \right). \label{pr:var2}
\end{align}
By stationarity, we have
\begin{align}
&\ep \left[ \frac{1}{\lam(\sq_{n+1})}\sum_{z \in 3^n\mathbb{Z}^d\cap \sq_{n+1}} \left(J \left(z+\sq_n,p,q\right)-J\left(\sq_{n+1},p,q\right) \right) :\lam(\sq_{n+1})>\del_{n+1}\right] \\
\leq&\frac{1}{\del_{n+1}} \ep \left[ \sum_{z \in 3^n\mathbb{Z}^d\cap \sq_{n+1}} \left(J \left(z+\sq_n,p,q\right)-J\left(\sq_{n+1},p,q\right) \right)\right]\leq \frac{C}{\del_{n+1}}\tau_n
\end{align}
Since $\left( \{\del_n\},\{M_n\} \right)$ is suppressive, it follows from \eqref{ieq:Jpq} and \eqref{ieq:sup} that
\begin{align}
&\ep \left[ \frac{1}{\lam(\sq_{n+1})}\sum_{z \in 3^n\mathbb{Z}^d\cap \sq_{n+1}} \left(J \left(z+\sq_n,p,q\right)-J\left(\sq_{n+1},p,q\right) \right) :\lam(\sq_{n+1})\leq\del_{n+1}\right] \\
\leq&\,C\,\ep \left[ \frac{\Lam(\sq_{n+1})}{\lam(\sq_{n+1})}+\frac{1}{\lam(\sq_{n+1})^2}+\frac{1}{\lam(\sq_{n+1})}:\lam(\sq_{n+1})\leq\del_{n+1}\right]\leq Ce^{-n}
\end{align}
Combining the four previous displays yields \eqref{pr:var}.

\textit{Step 2.} In this step, we use uniform decrease of maximal correlation and prove the existence of $\theta=\theta(d,r) \in (0,1)$ such that, for every $p,q \in B_1$ and $n \in \mathbb{N}$,
\begin{equation}
\text{Var} \left[ 3^{-d} \sum_{z \in 3^n\mathbb{Z}^d\cap \sq_{n+1}} \fint_{z+\sq_n}\nab v \left( \cdot,z+\sq_n,p,q \right) \right]^{1/2} \leq \theta \, \text{Var} \left[ \fint_{\sq_n} \nab v \left(\cdot,\sq_n,p,q\right) \right]^{1/2}. \label{pr:var2}
\end{equation} 
For the sake of simplicity, we denote $X_z:=\fint_{z+\sq_n}\nab v \left(\cdot,z+\sq_n,p,q \right)$ for each $z \in 3^n\mathbb{Z}^d$. Expanding the variance, we find that
\begin{equation}
\text{Var}\left[ 3^{-d} \sum_{z \in 3^n\mathbb{Z}^d\cap \sq_{n+1}} X_z\right] =3^{-2d} \sum_{z,z' \in 3^n\mathbb{Z}^d\cap \sq_{n+1}} \text{Cov} \left[ X_z,X_{z'} \right].
\end{equation}
Using the Cauchy-Schwarz inequality and stationarity yields
\begin{equation}
\left| \text{Cov} \left[ X_z,X_{z'} \right] \right| \leq \text{Var} \left[ X_z \right]^{1/2} \text{Var} \left[ X_{z'} \right]^{1/2} \leq \text{Var} \left[ X_0 \right].
\end{equation}
We see that each component of $X_z$ belongs to $\mathcal{L}^2 \left( \mathcal{F}_{z+\sq_n}\right)$. In the case when $z+\sq_n$ and $z'+\sq_n$ are neither identical nor neighbors, \eqref{as3} shows that
\begin{equation}
\left| \text{Cov} \left[ X_z,X_{z'} \right] \right| \leq r \text{Var} \left[ X_z \right]^{1/2} \text{Var} \left[ X_{z'} \right]^{1/2} \leq r \text{Var} \left[ X_0 \right].
\end{equation}
Since we can find at least one pair of subcubes $z+\sq_n$ and $z'+\sq_n$ which are neither identical nor neighbors, we obtain
\begin{equation}
\text{Var} \left[ 3^{-d} \sum_{z \in 3^n\mathbb{Z}^d\cap \sq_{n+1}} X_z\right]  \leq \left( 1-\frac{1-r}{3^{2d}} \right) \text{Var} [X_0].
\end{equation}
Thus, \eqref{pr:var2} is proved for $\theta:=\left( 1-\displaystyle\frac{1-r}{3^{2d}} \right)^{1/2}$.

\textit{Step 3.} 
We complete the arguments by iterating \eqref{pr:var} and \eqref{pr:var2}.  Fix $p,q \in B_1$ and denote
\begin{equation}
\sig_n:=\text{Var} \left[ \fint_{\sq_n} \nab v(\cdot,\sq_n,p,q) \right].
\end{equation}
Combining \eqref{pr:var} and \eqref{pr:var2}, it follows that, for every $n \in \mathbb{N}$,
\begin{equation}
\sig_{n+1}\leq \theta\sig_n+\frac{C}{\del_{n+1}^{1/2}}\tau_{n}^{1/2}+Ce^{-\ka' n}.
\end{equation}  
An iteration and decrease of $\{ \del_n \}$ yield
\begin{equation}
\sig_n \leq \theta^n \sig_0 + C\sum_{m=0}^{n-1} \theta^{n-m-1}e^{-\ka' m} +\frac{C}{\del_{n}^{1/2}}\sum_{m=0}^{n-1} \theta^{n-m-1} \tau_{m}^{1/2}. \label{pr:var3}
\end{equation}
Notice that $\sig_0 \leq C=C(L)$. Indeed, using Cauchy-Schwarz inequality, we have
\begin{equation}
\sig_0^2 \leq \ep\left[ \left| \fint_{\sq_0} \nab v(\cdot,\sq_0,p,q) \right|^2 \right] \leq \ep \left[ \fint_{\sq_0} \left| \nab v(\cdot,\sq_0,p,q) \right|^2 \right] .
\end{equation}
Hence, this inequality, \eqref{ieq:v2} and \eqref{ieq:Lplq2} lead to
\begin{equation}
\sig_0^2 \leq \ep \left[ \fint_{\sq_0} \left| \nab v(\cdot,\sq_0,p,q) \right|^2 \right] \leq \ep \left[ \frac{\Lam(\sq_0)}{\lam(\sq_0)} + \frac{1}{\lam(\sq_0)^2}+\frac{1}{\lam(\sq_0)}\right] \leq C.
\end{equation}
Squaring \eqref{pr:var3} and setting $\ka :=\ka' \land -\log\theta$, we get
\begin{align}
\sig_n^2&\leq 3\theta^{2n}\sig_0^2 +3\left( C\sum_{m=0}^{n-1}\theta^{n-m-1}e^{-\ka' m} \right)^2+3\left( \frac{C}{\del_n^{1/2}}\sum_{m=0}^{n-1} \theta^{n-m-1}  \tau_{m}^{1/2} \right)^2 \\
&\leq  Ce^{-\ka n}+\frac{C}{\del_n} \left( \sum_{m=0}^{n-1} e^{-\ka(n-m-1) } \tau_{m} \right).
\end{align}
The proof is complete.
\end{proof}

By applying the similar arguments as Proposition \ref{propprop:baa}, we have, for every $n \in \mathbb{N}$,
\begin{equation}
\ep \left[ \Lam \left( \sq_0 \right) \right]^{-1} \textup{\textsf{Id}} \leq \ep \left[ \aa_{*}^{-1} \left(\sq_n \right) \right] \leq \ep \left[ \aa_{*}^{-1} \left(\sq_0 \right) \right] \leq \ep \left[ \lam \left( \sq_0 \right)^{-1} \right] \textup{\textsf{Id}}.  \label{ieq:mu*n}
\end{equation}
\begin{dfn}
We define a deterministic matrix $\bar{\aa}_n$ by
\begin{equation}
\bar{\aa}_n:=\ep \left[ \aa_{*}^{-1} (\sq_n) \right]^{-1}. \label{def:ban}
\end{equation} 
\end{dfn}
By \eqref{ieq:mu*n}, we have the bound of $\bar{\aa}_n$ such that, for every $n \in \mathbb{N}$,
\begin{equation}
\ep[\lam(\square_0)^{-1}]^{-1} \textup{\textsf{Id}} \leq \bar{\aa}_0 \leq \bar{\aa}_n \leq \ep[\Lam(\square_0)] \textup{\textsf{Id}}. \label{ieq:ban}
\end{equation}
Since \eqref{supep0} and \eqref{ieq:ban}, the existence of suppressive sequences leads the existence of constants $c=c(L)>0$ and $C=C(L)<\infty$ such that
\begin{equation}
c \,\textup{\textsf{Id}} \leq \bar{\aa}_n \leq C\, \textup{\textsf{Id}}. \label{ieq:ban3}
\end{equation}
We will show that $\bar{\aa}_n$ converges to the homogenized coefficients $\bar{\aa}$. This is not trivial, since $\bar{\aa}_n$ is defined via $\aa_{*}(\sq_n)$, while $\bar{\aa}$ is defined via $\aa(\sq_n)$.

We denote some properties of $\bar{\aa}_n$. \eqref{lem:rq4} leads, for every $q \in \mathbb{R}^d$,
\begin{equation}
\ep \left[ \fint_{\sq_n} \nab v \left( \cdot,\sq_n,0,q \right) \right]= \bar{\aa}_n^{-1}q.
\end{equation}
On the other hand, since the function $v(\cdot,\sq_n,p,0)$ is belong to $-l_p+H_0^1(\sq_n)$, it follows that, for every $p \in \mathbb{R}^d$,
\begin{equation}
\fint_{\sq_n} \nab v(\cdot,\sq_n,p,0)=-p.
\end{equation}
Combining the previous two displays and \eqref{vli}, we have
\begin{equation}
\ep \left[ \fint_{\sq_n} \nab v \left( \cdot,\sq_n,p,q \right) \right]= \bar{\aa}_n^{-1}q-p, \label{epspv}%
\end{equation}
which is related to the quantity appeared in Lemma \ref{lemlem:varsa}.

We also note that 
\begin{equation}
\ep \left[ J \left( \sq_n,p,\bar{\aa}_np \right) \right] \leq \ep \left[ J \left( \sq_n,p,q \right) \right] \leq \ep \left[ J \left( \sq_n,p,\bar{\aa}_np \right) \right]+\frac{1}{2} \ep \left[ \lam \left(\sq_0 \right)^{-1} \right] \left| q-\bar{\aa}_np \right|^2. \label{ieq:jban}
\end{equation}
Indeed, it follows from \eqref{lem:rq2} that 
\begin{equation}
\ep \left[ J \left( \sq_n,p,q \right) \right] - \ep \left[ J \left( \sq_n,p,\bar{\aa}_np \right) \right] =\frac{1}{2} \left| \bar{\aa}_n^{-1/2} \left( q - \bar{\aa}_np \right) \right|^2.
\end{equation}
Hence, \eqref{ieq:ban} yields \eqref{ieq:jban}.

We next use the previous lemma and the multiscale Poincar\'{e} inequality to estimate the flatness of $v(\cdot, \sq_{n+1},p,,q)$ by $\tau_0,\tau_1,\dots,\tau_n$.


\begin{lem} \label{lemlem:flatv}
Let $\left( \{\del_n\},\{M_n\} \right)$ be suppressive. Then, there exist $\ka=\ka(d,r)>0$ and $C=C(d,r,L)<\infty$ such that, for every $n \in \mathbb{N}$ and $p,q \in B_1$,
\begin{equation}
\ep \left[ \fint_{\sq_{n+1}} \left| v(x,\sq_{n+1},p,q) - \left( \bar{\aa}_n^{-1}q-p \right) \cdot x \right|^2 dx \right] \leq \frac{C}{\del_{n+1}}3^{2n} \left( e^{-\ka n}+\sum_{k=0}^{n} e^{-\ka (n-k)}\tau_{k} \right). \label{lem:flatv}
\end{equation}
\end{lem}
\begin{proof}
\textit{Step.1.} We first use the multiscale Poincar\'{e} inequality. Since $v(\cdot,\sq_{n+1},p,0)$ is belong to $-l_p+H_0^1(\sq_{n+1})$ and the spatial average of $v(\cdot,\sq_{n+1},0,q)$ is zero, the multiscale Poincar\'{e} inequality (Lemma \ref{lemlem:mpi}) shows that
\begin{align}
&\ep \left[ \fint_{\sq_{n+1}} \left| v(x,\sq_{n+1},p,q) - \left( \bar{\aa}_n^{-1}q-p \right) \cdot x \right|^2 dx \right] \\
\leq&\,2 \left( \ep \left[ \fint_{\sq_{n+1}} \left| v(x,\sq_{n+1},p,0) +p \cdot x \right|^2 dx \right] +\ep \left[ \fint_{\sq_{n+1}} \left| v(x,\sq_{n+1},0,q) - \bar{\aa}_n^{-1}q\cdot x \right|^2 dx \right]\right) \\
\leq&\,C\,\ep \left[ \fint_{\sq_{n+1}} \left| \nab v(\cdot,\sq_{n+1},p,0) +p\right|^2 \right]+C\,\ep \left[ \fint_{\sq_{n+1}} \left| \nab v(\cdot,\sq_{n+1},0,q) -\bar{\aa}_n^{-1}q\right|^2 \right] \\
\quad&+C\, \ep \left[ \left( \sum_{m=0}^{n} 3^m \left( 3^{-(n+1-m)d}\sum_{y \in 3^m\mathbb{Z}^d\cap \sq_{n+1}} \left| \fint_{y+\sq_m} \nab v \left( \cdot,\sq_{n+1},0,q\right)-\bar{\aa}_n^{-1}q \right|^2 \right)^{1/2} \right)^{2} \right] \\
\quad&+C\,\ep\left[ \left( \sum_{m=0}^{n} 3^m \left( 3^{-(n+1-m)d}\sum_{y \in 3^m\mathbb{Z}^d\cap \sq_{n+1}} \left| \fint_{y+\sq_m} \nab v \left( \cdot,\sq_{n+1},p,0\right)+p \right|^2 \right)^{1/2} \right)^{2} \right]. \label{pr:vmpi}
\end{align}
We estimate the first and second terms on the right-hand side. Since $\left( \{\del_n\},\{M_n\} \right)$ is suppressive, it follows from \eqref{moJ}, \eqref{ieq:Jpq} and \eqref{ieq:Lplq2} that, for every $n \in \mathbb{N}$ and every $p,q \in B_1$,
\begin{equation}
\ep \left[ J(\sq_n,p,q)\right] \leq \ep \left[ J(\sq_0,p,q) \right] \leq \ep \left[ \frac{1}{2}\Lam(\sq_0)+ \frac{1}{2\lam(\sq_0)} + 1\right] \leq C. \label{ieq:ep0c}
\end{equation}  
From \eqref{ieq:v2}, \eqref{ieq:sup} and \eqref{ieq:ep0c}, we obtain, for every $n \in \mathbb{N}$ and every $p,q \in B_1$,
\begin{align}
&\ep \left[ \fint_{\sq_{n}} | \nab v(\cdot,\sq_{n},p,q) |^2 \right] \\
\leq&\, \ep \left[ \frac{2}{\lam(\sq_{n})} J(\sq_n,p,q):\lam(\sq_n)\leq \del_{n} \right] +\ep \left[ \frac{2}{\lam(\sq_{n})} J(\sq_n,p,q):\lam(\sq_n)> \del_{n} \right] \\
\leq&\,C\,\ep \left[ \frac{\Lam(\sq_{n})}{\lam(\sq_{n})}+\frac{1}{\lam(\sq_{n})^2}+\frac{1}{\lam(\sq_{n})}:\lam(\sq_{n})\leq\del_{n}\right]+\frac{C}{\del_n} \ep \left[ J(\sq_n,p,q) \right]  \\ 
\leq&\, C e^{-n}+\frac{C}{\del_n} \ep \left[ J(\sq_n,p,q) \right]. \label{ieq:nabv}
\end{align}
Combining \eqref{ieq:ep0c}, \eqref{ieq:nabv} and \eqref{ieq:ban3}, we obtain
\begin{align}
\ep \left[ \fint_{\sq_{n+1}} \left| \nab v(\cdot,\sq_{n+1},p,0) +p\right|^2 +\fint_{\sq_{n+1}} \left| \nab v(\cdot,\sq_{n+1},0,q) -\bar{\aa}_n^{-1}q\right|^2 \right] \leq&\,C e^{-n}+\frac{C}{\del_n}.  \label{pr:vmpi1}
\end{align}

\textit{Step 2.} In this step, we use the estimate of the variance of the spatial average of $\nab v(\cdot,\sq_n,p,q)$ to obtain the existence of $\ka'=\ka'(d,r)>0$ and $C=C(d,r,L)<\infty$ such that, for every $m \in \{ 1,2,\dots,n \}$ and every $p,q \in B_1$,
\begin{align}
3^{-(n+1-m)d}\sum_{y \in 3^m\mathbb{Z}^d\cap \sq_{n+1}} \ep &\left[ \left| \fint_{y+\sq_m} \nab v(\cdot,\sq_{n+1},p,q)-\bar{\aa}_n^{-1}q+p \right|^2 \right] \\
 &\leq\frac{C}{\del_{n+1}}\left( e^{- \ka' m} +\sum_{k=0}^{m-1} e^{\ka'(k+1-m)}\tau_k+\sum_{k=m}^{n} \tau_k \right).  \label{pr:flatst2}
\end{align}
By \eqref{epspv}, stationarity and Cauchy-Schwarz inequality, we have, for every $q \in B_1$ and $m$, $n\in\mathbb{N}$ with $m<n$,
\begin{align}
&\left| \bar{\aa}_n^{-1}q-\bar{\aa}_m^{-1}q \right|^2 \\
=& \left| \ep \left[ \fint_{\sq_n} \nab v(x,\sq_n,0,q)\,dx- \fint_{\sq_m}\nab v(x,\sq_m,0,q)\,dx\right]\right|^2 \\
=& \left| \ep \left[ 3^{-d(n-m)}\fint_{z+\sq_m} \nab v(x,\sq_n,0,q)\,dx-3^{-d(n-m)}\sum_{ z \in 3^m\mathbb{Z}^d\cap \sq_{n}} \fint_{z+\sq_n}\nab v(x,z+\sq_m,0,q)\,dx\right]\right|^2 \\
\leq&\, \ep \left[ 3^{-d(n-m)} \sum_{z \in 3^m\mathbb{Z}^d\cap \sq_{n}} \fint_{z+\sq_m} \left| \nab v(\cdot,\sq_n,0,q)-\nab v (\cdot,z+\sq_m,0,q) \right|^2 \right]. \label{anam-1}
\end{align}
From \eqref{def:lam} and \eqref{lem:qr2}, we have 
\begin{align}
&\,\ep \left[ 3^{-d(n-m)} \sum_{z \in 3^m\mathbb{Z}^d\cap \sq_{n}} \fint_{z+\sq_m} \left| \nab v(\cdot,\sq_n,p,q)-\nab v(\cdot,z+\sq_m,p,q) \right|^2 \right]\\
\leq  & \,\ep \left[ \frac{1}{\lam(\sq_n)} 3^{-d(n-m)}  \sum_{z \in 3^m\mathbb{Z}^d\cap \sq_{n}} \fint_{z+\sq_m}  \left|\aa^{\frac{1}{2}}\left(\nab v(\cdot,\sq_n,p,q)-\nab v(\cdot,z+\sq_m,p,q) \right) \right|^2 \right] \\
= & \, \ep \left[ \frac{2}{\lam(\sq_n)}  3^{-d(n-m)}\sum_{z \in 3^m\mathbb{Z}^d\cap \sq_{n}} \left( J(z+\sq_m,p,q)-J(\sq_n,p,q)\right) \right]. \label{ieq:parnabv}
\end{align}
From \eqref{anam-1} and \eqref{ieq:parnabv}, it follows that
\begin{equation}
\left| \bar{\aa}_n^{-1}q-\bar{\aa}_m^{-1}q \right|^2 \leq  \ep \left[ \frac{2}{\lam(\sq_n)}3^{-d(n-m)} \sum_{z \in 3^m\mathbb{Z}^d\cap \sq_{n}} \left( J(z+\sq_m,0,q)-J(\sq_n,0,q)\right) \right]. \label{ieq:agag1}
\end{equation}
By stationarity, it follows that
\begin{align}
&\ep \left[ \frac{2}{\lam(\sq_n)} 3^{-d(n-m)} \sum_{z \in 3^m\mathbb{Z}^d\cap \sq_{n}} \left( J(z+\sq_m,p,q)-J(\sq_n,p,q) \right) : \lam(\sq_n)>\del_n \right] \\
\leq&\frac{2}{\del_{n}} \ep \left[ 3^{-d(n-m)}\sum_{z \in 3^m\mathbb{Z}^d\cap \sq_{n}} \left(J \left(z+\sq_m,p,q\right)-J\left(\sq_{n},p,q\right) \right)\right]\leq \frac{C}{\del_{n}}\sum_{k=m}^{n-1}\tau_k. \label{ieq:agag2}
\end{align}
Since $\left( \{\del_n\},\{M_n\} \right)$ is suppressive, we see from \eqref{ieq:Jpq} and \eqref{ieq:sup} that
\begin{align}
&\ep \left[ \frac{2}{\lam(\sq_{n})}3^{-d(n-m)}\sum_{z \in 3^m\mathbb{Z}^d\cap \sq_{n}} \left(J \left(z+\sq_m,p,q\right)-J\left(\sq_{n},p,q\right) \right) :\lam(\sq_{n})\leq\del_{n}\right] \\
\leq&\,C\,\ep \left[ \frac{\Lam(\sq_{n})}{\lam(\sq_{n})}+\frac{1}{\lam(\sq_{n})^2}+\frac{1}{\lam(\sq_{n})}:\lam(\sq_{n})\leq\del_{n}\right]\leq Ce^{-n}. \label{ieq:agag3}
\end{align}
\eqref{ieq:parnabv}, \eqref{ieq:agag1}, \eqref{ieq:agag2} and \eqref{ieq:agag3} lead to
\begin{equation}
\ep \left[ 3^{-d(n-m)} \sum_{z \in 3^m\mathbb{Z}^d\cap \sq_{n}} \fint_{z+\sq_m} \left| \nab v(\cdot,\sq_{n},p,q)-\nab v(\cdot,z+\sq_m,p,q) \right|^2 \right] \leq Ce^{-n}+\frac{C}{\del_{n}}\sum_{k=m}^{n-1}\tau_k \label{ieq:epsc}
\end{equation}
and
\begin{equation}
\left| \bar{\aa}_n^{-1}q-\bar{\aa}_m^{-1}q \right|^2 \leq Ce^{-n}+\frac{C}{\del_{n}}\sum_{k=m}^{n-1}\tau_k. \label{ipq:acauchy}
\end{equation}
From the previous two inequalities, Lemma \ref{lemlem:varsa} and \eqref{epspv}, we deduce that there exist constants $\ka=\ka(d,r) \in (0, 1)$ and $C=C(d,r,L)>0$ such that
\begin{align}
&3^{-(n+1-m)d}\sum_{y \in 3^m\mathbb{Z}^d\cap \sq_{n+1}} \ep \left[ \left| \fint_{y+\sq_m} \nab v(\cdot,\sq_{n+1},p,q)-\bar{\aa}_n^{-1}q+p \right|^2 \right] \\
\leq&\,3\cdot3^{-(n+1-m)d}\sum_{y \in 3^m\mathbb{Z}^d \cap \sq_{n+1}} \ep \left[ \left| \fint_{y+\sq_m} \nab v(\cdot,\sq_{n+1},p,q)-\nab v (\cdot,y+\sq_m,p,q) \right|^2 \right] \\
&+3\cdot3^{-(n+1-m)d}\sum_{y \in 3^m\mathbb{Z}^d\cap \sq_{n+1}} \ep \left[ \left| \fint_{y+\sq_m} \nab v(\cdot,y+\sq_{m},p,q)-\bar{\aa}_m^{-1}q+p \right|^2 \right] \\
&+3 \left|\bar{\aa}_m^{-1}q-\bar{\aa}_n^{-1}q \right|^2 \\
\leq&\frac{C}{\del_{n+1}}\left(e^{- \ka m} +\sum_{k=0}^{m-1} e^{\ka(k+1-m)}\tau_k+\sum_{k=m}^{n} \tau_k \right).
\end{align}
Hence, \eqref{pr:flatst2} is proved.

\textit{Step 3.} For the sake of simplicity, we write
\begin{equation}
X_m:=3^{-(n+1-m)d} \sum_{y \in 3^m\mathbb{Z}^d\cap \sq_{n+1}}  \left| \fint_{y+\sq_m} \nab v(\cdot,\sq_{n+1},p,q)-\bar{\aa}_n^{-1}q+p \right|^2 .
\end{equation}
Using H\"{o}lder's inequality, we see that
\begin{equation}
\left(\sum_{m=0}^{n} 3^m X_m^{1/2} \right)^2 \leq \left( \sum_{m=0}^{n}3^m \right) \left( \sum_{m=0}^{n}3^mX_m \right) \leq \,C\,3^n \sum_{m=0}^{n} 3^mX_m.
\end{equation}
From \eqref{pr:flatst2} and $\ka \in (0,1)$, we obtain
\begin{align}
\ep \left[ \left(\sum_{m=0}^{n} 3^m X_m^{1/2} \right)^2 \right] \leq&\frac{C}{\del_{n+1}}3^n\sum_{m=0}^{n}3^m \left( e^{- \ka m} +\sum_{k=0}^{m-1} e^{\ka(k+1-m)}\tau_k+\sum_{k=m}^{n} \tau_k \right) \\
\leq&\frac{C}{\del_{n+1}}3^{2n}\left( e^{-\ka n} +\sum_{k=0}^{n-1} e^{-\ka(n-k)}\tau_k + \sum_{k=0}^{n} 3^{-(n-k)} \tau_k \right) \\
\leq&\frac{C}{\del_{n+1}}3^{2n}\left( e^{-\ka n} +\sum_{k=0}^{n} e^{-\ka(n-k)}\tau_k \right). \label{ieq:ahah1}
\end{align}
From \eqref{ieq:ahah1}, \eqref{pr:vmpi} and \eqref{pr:vmpi1}, we have \eqref{lem:flatv}.
\end{proof}

Combining the previous estimate of the flatness of $v(\cdot, \sq_{n+1},p,q)$ and the Caccioppoli inequality, we now control $J(\sq_n,p,\bar{\aa}_np)$ by $\tau_0,\tau_1,\dots,\tau_n$.

\begin{lem}\label{lemlem:epJn}
Let $\left( \{\del_n\},\{M_n\} \right)$ be suppressive. Then, there exist $\ka=\ka(d,r)>0$ and $C=C(d,r,L)<\infty$ such that, for every $n \in \mathbb{N}$ and $p\in B_1$,
\begin{equation}
\ep \left[ J(\sq_n,p,\bar{\aa}_np) \right] \leq C\, \frac{M_{n+1}^3}{\del_{n+1}^3}\left( e^{-\ka n} + \sum_{k=0}^n 3^{-\ka(n-k)} \tau_m \right). \label{lem:epJn}
\end{equation}
\end{lem}
\begin{proof}
Fix $p \in B_1$. Choose $\ka \in (0,\frac{1}{2})$ so that Lemma \ref{lemlem:flatv} holds.

\textit{Step 1.} In this step, we will show that
\begin{equation}
\ep \left[ \fint_{\sq_n}| \nab v(\cdot,\sq_{n+1},p,\bar{\aa}_np) |^2 \right] \leq C\frac{M_{n+1}^2}{\del_{n+1}^3} \left( e^{-\ka n} +\sum_{k=0}^{n} e^{-\ka(n-k)}\tau_k \right). \label{pr:jstep1}
\end{equation}
Notice that $|\bar{\aa}_np| \leq C=C(L)$ holds by \eqref{ieq:ban3}. From \eqref{ieq:v2}, we have 
\begin{equation}
\fint_{\sq_n} | \nab v(\cdot,\sq_{n+1},p,\bar{\aa}_np)|^2 \leq C \fint_{\sq_{n+1}} | \nab v(\cdot,\sq_{n+1},p,\bar{\aa}_np)|^2 \leq C \left( \frac{\Lam(\sq_{n+1})}{\lam(\sq_{n+1})}+\frac{1}{\lam(\sq_{n+1})^2}+\frac{1}{\lam(\sq_{n+1})} \right).
\end{equation}
From this and \eqref{ieq:sup}, we get 
\begin{equation}
\ep \left[ \fint_{\sq_n} | \nab v(\cdot,\sq_{n+1},p,\bar{\aa}_np)|^2 :\left\{ \lam\left(\sq_n\right)\leq \del_n\right\} \cup \left\{ \Lam\left(\sq_n\right)\geq M_n\right\}\right] \leq Ce^{-n}.
\end{equation}
Applying the Caccioppoli inequality (Lemma \ref{lem:cacci}), we find that
\begin{equation}
\ep \left[ \fint_{\sq_n} | \nab v(\cdot,\sq_{n+1},p,\bar{\aa}_np)|^2 :\lam\left(\sq_n\right)>\del_n , \Lam\left(\sq_n\right)<M_n\right] \leq \frac{C}{3^{2n}} \frac{M_{n+1}^2}{\del_{n+1}^2}\ep \left[ \fint_{\sq_{n+1}}|v(\cdot,\sq_{n+1},p,\bar{\aa}_np) |^2 \right].
\end{equation}
Since $|\bar{\aa}_np| \leq C$ holds, it follows from Lemma \ref{lemlem:flatv} that
\begin{equation}
\ep \left[ \fint_{\sq_{n+1}}|v(\cdot,\sq_{n+1},p,\bar{\aa}_np) |^2 \right] \leq \frac{C}{\del_{n+1}} 3^{2n} \left( e^{-\ka n} +\sum_{k=0}^{n} e^{-\ka(n-k)}\tau_k \right).
\end{equation}
Combining the above inequalities yield \eqref{pr:jstep1}.

\textit{Step 2.} Fix $p,q \in B_1$ and define $v:=v(\cdot,\sq_n,p,q)$ and $w:=v(\cdot,\sq_{n+1},p,q)$. We will prove that 
\begin{equation}
\ep \left[ J(\sq_n,p,q) \right] \leq C\,\ep \left[ \fint_{\sq_n} \nab w \cdot \aa \nab w \right] + C \,\frac{M_{n+1}}{\del_{n+1}} \left( e^{-\ka n}+ \tau_n^{1/2}\ep\left[ J(\sq_n,p,q) \right]^{1/2}\right). \label{pr:jstep2}
\end{equation}
Using \eqref{lem:Jq}, we find that 
\begin{equation}
J(\sq_n,p,q)=\fint_{\sq_n} \frac{1}{2} \nab v \cdot \aa \nab v=\fint_{\sq_n} \frac{1}{2} \nab w \cdot \aa \nab w - \fint_{\sq_n}\frac{1}{2} (\nab v-\nab w) \cdot \aa (\nab v+\nab w). \label{pr:nabnabJ}
\end{equation} 
By H\"{o}lder's inequality, we get
\begin{equation}
\left| \fint_{\sq_n} (\nab v -\nab w)\cdot \aa ( \nab v +\nab w) \right| \leq C \Lam(\sq_n) \| \nab v-\nab w\|_{\ll^2(\sq_n)}\left( \| \nab v \|_{\ll^2(\sq_n)}+\| \nab w \|_{\ll^2(\sq_n)}\right). \label{ieq:aiai1}
\end{equation}
From \eqref{ieq:v2}, we have 
\begin{equation}
\| \nab w \|_{\ll^2(\sq_n)} \leq C\| \nab w \|_{\ll^2(\sq_{n+1})}\leq C \left( \frac{\Lam(\sq_{n+1})}{\lam(\sq_{n+1})}+\frac{1}{\lam(\sq_{n+1})^2}+\frac{1}{\lam(\sq_{n+1})} \right)^{\frac{1}{2}}. \label{ieq:aiai2}
\end{equation}
The same estimate holds for $\| \nab v \|_{\ll^2(\sq_n)} $ and $\| \nab v-\nab w \|_{\ll^2(\sq_n)} $. From \eqref{ieq:aiai1}, \eqref{ieq:aiai2} and \eqref{ieq:sup}, we find that
\begin{align}
&\ep \left[ \left| \fint_{\sq_n} (\nab v -\nab w)\cdot \aa ( \nab v +\nab w) \right| \right] \\
\leq&\,C\,\ep \left[ \Lam(\sq_n) \| \nab v-\nab w\|_{\ll^2(\sq_n)}\left( \| \nab v \|_{\ll^2(\sq_n)}+\| \nab w \|_{\ll^2(\sq_n)}\right) ;\Lam(\sq_{n+1})<M_{n+1} \right] \\
&+C\,\ep \left[  \frac{\Lam(\sq_{n+1})^2}{\lam(\sq_{n+1})}+\frac{\Lam(\sq_{n+1})}{\lam(\sq_{n+1})^2}+\frac{\Lam(\sq_{n+1})}{\lam(\sq_{n+1})} :\Lam(\sq_{n+1})\geq M_{n+1} \right] \\
\leq&\,CM_{n+1}\ep \left[ \| \nab v- \nab w \|_{\ll^2(\sq_n)}^2 \right]^{\frac{1}{2}} \left( \ep \left[ \| \nab v \|_{\ll^2(\sq_n)}^2 \right]^{\frac{1}{2}}+\ep \left[ \|\nab w \|_{\ll^2(\sq_{n+1})}^2 \right]^{\frac{1}{2}} \right) +Ce^{-n}. \label{ieq:ajaj1}
\end{align}
By \eqref{ieq:nabv} and \eqref{moJ}, we have
\begin{equation}
\ep \left[ \| \nab v \|_{\ll^2(\sq_n)}^2 \right]^{\frac{1}{2}} \leq Ce^{-\frac{1}{2}n} +\frac{C}{\del_{n}^{1/2}}\ep \left[ J(\sq_n,p,q) \right]^{\frac{1}{2}} \leq Ce^{-\frac{1}{2}n} +\frac{C}{\del_{n+1}^{1/2}}\ep \left[ J(\sq_n,p,q) \right]^{\frac{1}{2}}  , \label{ieq:ajaj2}
\end{equation}
and
\begin{equation}
\ep \left[ \|\nab w \|_{\ll^2(\sq_{n+1})}^2 \right]^{\frac{1}{2}} \leq Ce^{-\frac{1}{2}n} +\frac{C}{\del_{n+1}^{1/2}}\ep \left[ J(\sq_{n+1},p,q) \right]^{\frac{1}{2}} \leq Ce^{-\frac{1}{2}n} +\frac{C}{\del_{n+1}^{1/2}}\ep \left[ J(\sq_n,p,q) \right]^{\frac{1}{2}}. \label{ieq:ajaj3}
\end{equation}
Applying \eqref{ieq:epsc}, we obtain
\begin{align}
&\ep \left[ \| \nab v-\nab w \|_{\ll^2(\sq_n)}^2 \right] \\
\leq &\,C\, \ep \left[ 3^{-d} \sum_{z \in 3^n\mathbb{Z}^d\cap \sq_{n+1}} \fint_{z+\sq_n} \left| \nab v(\cdot,z+\sq_n,p,q)-\nab v(\cdot,\sq_{n+1},p,q) \right|^2 \right] \\
\leq &\,Ce^{-n} +\frac{C}{\del_{n+1}}\tau_{n}. \label{ieq:ajaj4}
\end{align}
From \eqref{ieq:ep0c}, we have
\begin{equation}
\tau_n \leq \sup_{p,q \in B_1} \ep [ J(\sq_n,p,q) ] \leq C. \label{ieq:ajaj5}
\end{equation}
Combining the previous five inequalities, we obtain
\begin{align}
&\ep \left[ \left| \fint_{\sq_n} (\nab v -\nab w)\cdot \aa ( \nab v +\nab w) \right| \right] \\
\leq&\,CM_{n+1} \left( e^{-\frac{1}{2}n}+\frac{C}{\del_{n+1}^{1/2}}\tau_{n}^{1/2}\right) \left( e^{-\frac{1}{2}n}+\frac{C}{\del_{n+1}^{1/2}} \ep \left[ J(\sq_n,p,q) \right]^{\frac{1}{2}} \right)+Ce^{-n} \\
\leq&\, CM_{n+1}  \left( e^{-n}+\frac{C}{\del_{n+1}^{1/2}}e^{-\frac{1}{2}n} +\frac{C}{\del_{n+1}}\tau_{n}^{1/2}\ep \left[ J(\sq_n,p,q) \right]^{1/2}\right) +Ce^{-n} \\
\leq&\, C \,\frac{M_{n+1}}{\del_{n+1}} \left( e^{-\frac{1}{2} n}+ \tau_n^{1/2}\ep\left[ J(\sq_n,p,q) \right]^{1/2}\right). \label{ieq:ajaj6}
\end{align}
Thus, \eqref{ieq:ajaj6} and \eqref{pr:nabnabJ} lead to \eqref{pr:jstep2}.

\textit{Step 3.} In this step, we finish the proof. Since $|\bar{\aa}_np| \leq C$ holds, it follows from \eqref{pr:jstep1}, \eqref{lem:Jq}, \eqref{ieq:Jpq} and \eqref{ieq:sup} that
\begin{align}
&\ep \left[ \fint_{\sq_n} \nab v(\cdot,\sq_{n+1},p,\bar{\aa}_np) \cdot \aa \nab v(\cdot,\sq_{n+1},p,\bar{\aa}_np) \right] \\
\leq& \,\ep \left[ \fint_{\sq_{n}}  \nab v(\cdot,\sq_{n+1},p,\bar{\aa}_np)  \cdot \aa  \nab v(\cdot,\sq_{n+1},p,\bar{\aa}_np)  :\Lam(\sq_{n+1}) < M_{n+1}\right] \\
&+C\,\ep \left[ \fint_{\sq_{n+1}}  \nab v(\cdot,\sq_{n+1},p,\bar{\aa}_np)  \cdot \aa  \nab v(\cdot,\sq_{n+1},p,\bar{\aa}_np)  :\Lam(\sq_{n+1}) \geq M_{n+1}\right] \\
\leq&\,M_{n+1}\ep \left[ \fint_{\sq_{n}} | \nab v(\cdot,\sq_{n+1},p,\bar{\aa}_np)  |^2 \right]+C \,\ep \left[ \Lam(\sq_{n+1})+\frac{1}{\lam(\sq_{n+1})}+1:\Lam(\sq_{n+1}) \geq M_{n+1} \right] \\
\leq&\,C\,\frac{M_{n+1}^3}{\del_{n+1}^3} \left( e^{-\ka n} +\sum_{k=0}^{n} e^{-\ka(n-k)}\tau_k \right). 
\end{align}
Combining this inequality and \eqref{pr:jstep2} and using Young's inequality, we obtain
\begin{align}
\ep \left[ J(\sq_n,p,\bar{\aa}_np) \right] &\leq C \,\frac{M_{n+1}}{\del_{n+1}}\tau_n^{\frac{1}{2}}\ep\left[ J(\sq_n,p,\bar{\aa}_np) \right]^{\frac{1}{2}}+C\,\frac{M_{n+1}^3}{\del_{n+1}^3} \left( e^{-\ka n} +\sum_{k=0}^{n} e^{-\ka(n-k)}\tau_k \right) \\
&\leq \frac{1}{2} \ep \left[ J(\sq_n,p,\bar{\aa}_np) \right]+C\,\frac{M_{n+1}^2}{\del_{n+1}^2}\tau_n+C\,\frac{M_{n+1}^3}{\del_{n+1}^3} \left( e^{-\ka n} +\sum_{k=0}^{n} e^{-\ka(n-k)}\tau_k \right).
\end{align}
Arranging this implies \eqref{lem:epJn}. 
\end{proof}

\begin{proof}[Proof of Theorem \ref{thm:coa}]
We set the quantity
\begin{equation}
F_n:=\sum_{i=1}^{d} \ep \left[ J(\sq_n,e_i,\bar{\aa}_ne_i) \right].
\end{equation} 
We note that there exists a constant $C=C(d)< \infty$ such that, for every positive definite matrix $A \in \mathbb{R}^{d \times d}$,
\begin{equation}
\sup_{p \in B_1} p \cdot Ap \leq C\sum_{i=1}^d e_i \cdot Ae_i. \label{ieq:posequ}
\end{equation}
\textit{Step 1.} We first prove the estimate
\begin{equation}
F_n \leq C \exp\left(-cn^{1-3\alpha}\right). \label{pr:fst1}
\end{equation}
From \eqref{ieq:jban}, it follows that
\begin{equation}
F_{n+1} \leq \sum_{i=1}^d \ep \left[ J(\sq_{n+1},e_i,\bar{\aa}_{n}e_i) \right]. \label{ieq:fj}
\end{equation}
From \eqref{lem:Jqd}, \eqref{momu} and \eqref{momu*}, the mapping $p \mapsto \ep \left[ \mu(\sq_n,p)\right] - \ep \left[ \mu(\sq_{n+1},p)\right]$ and $p \mapsto \ep \left[ \mu_*(\sq_n,\bar{\aa}_{n}p)\right] - \ep \left[ \mu_*(\sq_{n+1},\bar{\aa}_{n}p)\right]$ are positive quadratic forms. By using \eqref{ieq:fj}, \eqref{def:J}, \eqref{ieq:posequ} and \eqref{ieq:ban3}, we have the existence of $c>0$ such that
\begin{align}
F_n-F_{n+1}&\geq \sum_{i=1}^{d} \left( \ep \left[ J(\sq_n,e_i,\bar{\aa}_ne_i) \right] -\ep \left[ J(\sq_{n+1},e_i,\bar{\aa}_ne_i) \right]  \right) \\
&=\sum_{i=1}^d \left( \ep \left[ \mu(\sq_n,e_i) \right]-\ep \left[ \mu(\sq_{n+1},e_i) \right] +\ep \left[ \mu_*(\sq_n,\bar{\aa}_n e_i) \right] -\ep \left[ \mu_*(\sq_{n+1},\bar{\aa}_n e_i) \right]\right) \\
&\geq c \left( \sup_{p \in B_1} \left( \ep \left[ \mu(\sq_n,p) \right]-\ep \left[ \mu(\sq_{n+1},p) \right]\right) +\sup_{q \in \bar{\aa}_n B_1} \left( \ep \left[ \mu_*(\sq_n,q) \right]-\ep \left[ \mu_*(\sq_{n+1},q) \right]\right)\right) \\
&\geq c \left( \sup_{p \in B_1} \left( \ep \left[ \mu(\sq_n,p) \right]-\ep \left[ \mu(\sq_{n+1},p) \right]\right) +\sup_{q \in B_1} \left( \ep \left[ \mu_*(\sq_n,q) \right]-\ep \left[ \mu_*(\sq_{n+1},q) \right]\right)\right) \\
&\geq c\tau_n . \label{ieq:ftau}
\end{align}
To control $F_n$, we will estimate the quantity
\begin{equation}
G_n:=e^{-\frac{\ka}{2}n} \sum_{m=0}^{n} e^{\frac{\ka}{2}m}F_m.
\end{equation}
Since $ \left( \{\del_n\}, \{M_n\} \right)$ is suppressive, it follows from \eqref{ieq:Jpq}, \eqref{ieq:ban3} and \eqref{ieq:Lplq2} that $F_0\leq C$.
From \eqref{ieq:ftau} and $F_0\leq C$, it follows that
\begin{equation}
G_n-G_{n+1}=e^{-\frac{\ka}{2}n}\sum_{m=0}^{n}e^{\frac{\ka}{2}m}(F_m-F_{m+1})-Ce^{-\frac{\ka}{2}n}F_0 \geq c\left( e^{-\frac{\ka}{2}n}\sum_{m=0}^{n}e^{\frac{\ka}{2}m}\tau_m-Ce^{-\frac{\ka}{2}n}\right).
\end{equation}
Lemma \ref{lemlem:epJn} and $\displaystyle\frac{M_n}{\del_n} \leq (n+1)^{\alpha}$ lead to
\begin{align}
G_n&\leq C(n+1)^{3\alpha}e^{-\frac{\ka}{2}n}\sum_{m=0}^{n}e^{\frac{\ka}{2}m} \left( e^{-\ka m}+ \sum_{k=0}^{m} e^{-\ka(m-k)}\tau_k \right) \\
&\leq C(n+1)^{3\alpha}\left( e^{-\frac{\ka}{2}n} +e^{-\frac{\ka}{2}n} \sum_{k=0}^n \sum_{m=k}^n e^{\frac{\ka}{2}(2k-m)} \tau_k \right) \\
&\leq C(n+1)^{3\alpha}\left( e^{-\frac{\ka}{2}n}+ e^{-\frac{\ka}{2}n}\sum_{k=0}^{n}e^{\frac{\ka}{2}k}\tau_{k}\right).
\end{align}
Using $\tau_n\geq0$ and comparing the previous two inequalities yield
\begin{align}
G_{n+1}\leq G_n+Ce^{-\frac{\ka}{2}n} \leq&\, C(n+1)^{3\alpha}\left( e^{-\frac{\ka}{2}n}+ e^{-\frac{\ka}{2}n}\sum_{k=0}^{n}e^{\frac{\ka}{2}k}\tau_{k}\right) \\
\leq&\, C(n+1)^{3\alpha}\left( \left(G_{n}-G_{n+1} \right)+e^{-\frac{\ka}{2}n}\right).
\end{align} 
Rearranging this shows that, for $\theta_n:=1-\displaystyle\frac{1}{C(n+1)^{3\alpha}}\;(n \in \mathbb{N})$,
\begin{equation}
G_{n+1} \leq \theta_nG_n+Ce^{-\frac{\ka}{2}n}. \label{ieq:gr}
\end{equation}
We make $C$ in the definition of $\theta_n$ large enough so that $e^{-\ka/2}\leq \theta_0 \leq \theta_n$ holds. An iteration of \eqref{ieq:gr} gives
\begin{equation}
G_n \leq \left( \prod_{k=0}^{n-1} \theta_k \right)G_0+C\sum_{k=1}^{n} \left\{ \left(\prod_{l=n-k+1}^{n-1} \theta_l \right)e^{-\frac{\ka}{2}(n-k)}\right\}. \label{ieq:gi}
\end{equation}
From $e^{-\ka/2}\leq \theta_l$ and $1-x \leq e^{-x}$, it follows that
\begin{equation}
\left(\prod_{l=n-k+1}^{n-1} \theta_l \right)e^{-\frac{\ka}{2}(n-k)} \leq \prod_{k=0}^{n-1} \theta_k \leq  \exp\left( -\sum_{k=1}^{n-1}\frac{1}{Ck^{3\alpha}+1}\right) \leq \exp\left( -c n^{1-3\alpha}\right),
\end{equation}
and therefore from \eqref{ieq:gi} and $G_0 \leq F_0 \leq C$ that
\begin{equation}
G_n\leq C(n+1)\exp\left( -c n^{1-3\alpha}\right).
\end{equation}
By appropriately reducing the value of $c$, we conclude that
\begin{equation}
G_n\leq  C \exp \left( -c n^{1-3\alpha}\right).
\end{equation}
Hence, from $F_n\leq G_n$, \eqref{pr:fst1} follows.

\textit{Step 2.} To complete the proof of the theorem, it remains to get the estimate that replaces $F_n$ in \eqref{pr:fst1} with $\left| \aa(\sq_n)-\bar{\aa}\right|$. To do so, we show that $\bar{\aa}_n$ converges to $\bar{\aa}$ by using the estimate obtained in the previous step. From \eqref{ipq:acauchy}, \eqref{ieq:ftau} and \eqref{pr:fst1}, it follows that, for every $p \in B_1$ and $n \in \mathbb{N}$,
\begin{align}
\left| \bar{\aa}_{n+1}^{-1}p- \bar{\aa}_n^{-1}p\right|^2 \leq Ce^{-n} +C\frac{\tau_n}{\del_{n+1}} \leq &\,Ce^{-n}  +C(n+1)^{\alpha}(F_n-F_{n+1}) \\
\leq& \,Ce^{-n}+Cn^{\alpha} \exp \left( -cn^{1-3\alpha} \right).
\end{align}
Thus, by appropriately reducing the value of $c$, we obtain
\begin{equation}
\left| \bar{\aa}_{n+1}^{-1}- \bar{\aa}_n^{-1}\right| \leq C n^{-3\alpha}\exp \left( -c n^{1-3\alpha}\right). \label{ststst1}
\end{equation}
We note that
\begin{equation}
\sum_{k=n+1}^{\infty}k^{-3\alpha}\exp \left( -ck^{1-3\alpha}\right) \leq \int_{n}^{\infty} x^{-3\alpha} \exp \left( -cx^{1-3\alpha}\right)dx =C\exp \left( -cn^{1-3\alpha}\right). \label{ststst2}
\end{equation}
Combining \eqref{ststst1} and \eqref{ststst2}, we obtain, for every $n,m \in \mathbb{N}$ with $n<m$,
\begin{equation}
\left| \bar{\aa}_{m}^{-1}- \bar{\aa}_n^{-1}\right| \leq \sum_{k=n}^{\infty}  \left| \bar{\aa}_{k+1}^{-1}- \bar{\aa}_k^{-1}\right| \leq \sum_{k=n}^{\infty} k^{-3\alpha}\exp \left( -ck^{1-3\alpha}\right) \leq C\exp \left( -cn^{1-3\alpha}\right).   \\
\end{equation}
Since $|\bar{\aa}_n|,|\bar{\aa}_m|\leq C$, we have, for every $n,m \in \mathbb{N}$ with $n<m$,
\begin{equation}
\left| \bar{\aa}_n- \bar{\aa}_m\right|= \left| \bar{\aa}_n\left(\bar{\aa}_m^{-1}- \bar{\aa}_n^{-1}\right)\bar{\aa}_m\right| \leq C \exp \left( -c n^{1-3\alpha}\right).
\end{equation}
Thus $\{\bar{\aa}_n\}$ is a Cauchy sequence and there exists a symmetric matrix $\widetilde{\aa} \in \mathbb{R}^{d \times d}$ such that
\begin{equation}
\left| \bar{\aa}_n- \widetilde{\aa}\right| \leq C \exp \left( -c n^{1-3\alpha}\right).
\end{equation}
The previous inequality, \eqref{ieq:jban} and \eqref{pr:fst1} lead to 
\begin{equation}
\ep \left[ J(\sq_n,e_i,\widetilde{\aa}e_i) \right] \leq \ep \left[ J(\sq_n,e_i,\bar{\aa}_ne_i) \right]+C\left| \bar{\aa}_n- \widetilde{\aa}\right|^2 \leq C \exp \left( -c n^{1-3\alpha}\right).
\end{equation}
From \eqref{ieJ}, \eqref{lem:rq2} and \eqref{ieq:posequ}, we have
\begin{equation}
\ep \left[ \sup_{p \in B_1} J(\sq_n,p,\widetilde{\aa}p) \right] \leq C \,\ep  \left[ \sum_{i=1}^d J(\sq_n,e_i,\widetilde{\aa}e_i) \right] \leq  C \exp \left( -c n^{1-3\alpha}\right). \label{ieq:Jtila}
\end{equation}
Using Lemma \ref{lemlem:|J}, we obtain that
\begin{equation}
\ep \left[ |\aa(\sq_n)-\widetilde{\aa}|^2 \right] \leq  C \,\ep \left[ \Lam(\sq_n) \sup_{p \in B_1} J(\sq_n,p,\bar{\aa}p) \right].
\end{equation}
Since $\{\bar{\aa}_n\}$ converges to $\widetilde{\aa}$, it follows from \eqref{ieq:ban3} that $|\widetilde{\aa}| \leq C$. From this, \eqref{ieq:Jpq} and \eqref{ieq:Lplq}, it follows that
\begin{align}
&\ep \left[ \Lam(\sq_n) \sup_{p \in B_1} J(\sq_n,p,\widetilde{\aa}p) : \Lam(\sq_n) \geq M_{n} \right] \\
\leq & \, C \, \ep \left[ \Lam(\sq_n)^2+\frac{\Lam(\sq_n)}{\lam(\sq_{n})}+\Lam(\sq_n) : \Lam(\sq_n) \geq M_{n}  \right] \leq C e^{-n}.
\end{align}
Since $M_n \leq (n+1)^\alpha$, \eqref{ieq:Jtila} gives
\begin{align}
&\ep \left[ \Lam(\sq_n) \sup_{p \in B_1} J(\sq_n,p,\widetilde{\aa}p) : \Lam(\sq_n) < M_{n} \right] \\
\leq &\, M_{n} \, \ep \left[ \sup_{p \in B_1} J(\sq_n,p,\widetilde{\aa}p) \right] \leq (n+1)^{\alpha} \exp \left( -c n^{1-3\alpha}\right) \leq C \exp \left( -c n^{1-3\alpha}\right).
\end{align}
Combining the previous three inequalities, we obtain
\begin{equation}
\ep \left[ |\aa(\sq_n)-\widetilde{\aa}|^2 \right] \leq C \exp \left( -c n^{1-3\alpha}\right). \label{conaa2}
\end{equation}
Since $\widetilde{\aa}-\aa(\sq_n)$ is symmetric, by using Jensen's inequality, we have
\begin{align}
\ep \left[ \left| \widetilde{\aa}-\aa(\sq_n)\right|^2 \right] &\geq \ep\left[ \sup_{p \in B_1} \left| p \cdot \left(\widetilde{\aa}-\aa(\sq_n)\right)p\right|^2 \right] \\
&\geq  \sup_{p \in B_1} \ep \left[ \left| p \cdot \left(\widetilde{\aa}-\aa(\sq_n)\right)p\right|^2 \right]\geq \sup_{p \in B_1} \left|  p\cdot \widetilde{\aa}p -\ep \left[ p \cdot \aa(\sq_n)p\right]\right|^2 .
\end{align}
The previous two inequalities and \eqref{def:mu2} show that $2\ep \left[\mu(\sq_n,p) \right]$ converge to $p \cdot \widetilde{\aa}p$ as $n \to \infty$. On the other hand, since $2\ep \left[\mu(\sq_n,p) \right]$ converge to $p \cdot \bar{\aa}p$ as $n \to \infty$ from the definition of $\bar{\aa}$, it follows that $p \cdot \widetilde{\aa}p=p \cdot \bar{\aa}p$ for every $p \in B_1$ and thus $\widetilde{\aa}=\bar{\aa}$. Therefore \eqref{conaa2} leads to \eqref{thm:coaa}.
\end{proof}

\section{Homogenization of the Dirichlet Problem}

\subsection{Weak Convergence of Gradients and Fluxes}
For each $n \in \mathbb{N}$, we define a random variable
\begin{equation}
\Omega(n):= \left( \sum_{m=0}^{n} 3^{-(n-m)} \left( 3^{-(n-m)d} \sum_{z \in  3^m\mathbb{Z}^d\cap \sq_{n}} | \aa(z+\sq_m)-\bar{\aa}|\right)^{\frac{1}{2}}\right)^2.
\end{equation}

We now use the multiscale Poincar\'{e} inequality to see weak convergence of the flux and gradient of $v(\cdot,\sq_n,p)$ by $\Om(n)$.
\begin{prop} \label{propprop:ieqfl}
There exists $C=C(d)< \infty$ such that, for every $n \in \mathbb{N}$ and $p \in B_1$,
\begin{equation}
\left\| \nab v (\cdot,\sq_n,p) -p\right\|_{\hh^{-1}(\sq_n)}^2 \leq C\,\frac{\Lam(\sq_n)}{\lam(\sq_n)}+C3^{2n}\frac{1}{\lam(\sq_n)}\Om(n) \label{prop:ieqfl}
\end{equation}
and
\begin{equation}
\left\| \aa\nab v (\cdot,\sq_n,p) -\bar{\aa}p\right\|_{\hh^{-1}(\sq_n)}^2 \leq C\left( \frac{\Lam(\sq_n)^3}{\lam(\sq_n)}+|\bar{\aa}|^2\right)+C3^{2n}\left( \frac{\Lam(\sq_n)^2}{\lam(\sq_n)}+|\bar{\aa}|\right)\Om(n). \label{prop:ieqde}
\end{equation}
\end{prop}
\begin{proof}
We fix $p \in B_1$, $n \in \mathbb{N}$ and denote $v:=v(\cdot,\sq_n,p)$. By using the multiscale Poincar\'{e} inequality (Proposition \ref{lemlem:mpi}), we have
\begin{align}
\|\nab v -p &\|_{\hh^{-1}(\sq_n)}^2 \leq \,C\, \|\nab v-p \|_{\ll^2(\sq_n)}^2 \\
&+C3^{2n}\left( \sum_{m=0}^{n-1} 3^{-(n-m)}\left( 3^{-(n-m)d}\sum_{z\in3^m\mathbb{Z}^d\cap \sq_{n}} \left| \fint_{z+\sq_m}\left(\nab v-p\right) \right|^2\right)^{\frac{1}{2}}\right)^2 \label{pr:wc3}
\end{align}
and
\begin{align}
\|\aa\nab v -\bar{\aa}p &\|_{\hh^{-1}(\sq_n)}^2 \leq \,C\, \|\aa\nab v-\bar{\aa}p \|_{\ll^2(\sq_n)}^2 \\
&+C3^{2n}\left( \sum_{m=0}^{n-1} 3^{-(n-m)}\left( 3^{-(n-m)d}\sum_{z\in3^m\mathbb{Z}^d\cap \sq_{n}} \left| \fint_{z+\sq_m}\left(\aa\nab v-\bar{\aa}p\right) \right|^2\right)^{\frac{1}{2}}\right)^2. \label{pr:wc4}
\end{align}
Since $v(\cdot,\sq_n,p)=-v(\cdot,\sq_n,p,0)$, it follows from \eqref{ieq:v2} that
\begin{equation}
\fint_{\sq_n} |\nab v|^2 \leq \frac{\Lam(\sq_n)}{\lam(\sq_n)} \quad \text{and} \quad \fint_{\sq_n} |\aa\nab v|^2 \leq \Lam(\sq_n)^2 \fint_{\sq_n} |\nab v|^2 \leq \frac{\Lam(\sq_n)^3}{\lam(\sq_n)}.
\end{equation}
Thus, triangle inequality shows that 
\begin{equation}
\|\nab v-p \|_{\ll^2(\sq_n)}^2\leq \frac{\Lam(\sq_n)}{\lam(\sq_n)} \quad \text{and} \quad \|\aa\nab v-\bar{\aa}p \|_{\ll^2(\sq_n)}^2\leq \frac{\Lam(\sq_n)^3}{\lam(\sq_n)}+| \bar{\aa} |^2. \label{pr:wc5}
\end{equation}
It remains to estimate the last terms. For $m \in \{ 0,1,2 \dots n-1\}$, we define a function $v_m \in H^1(\sq_n)$ which satisfies, for every $z \in 3^m \mathbb{Z}^d \cap \sq_n$,
\begin{equation}
v_m(x):=v(x,z+\sq_m,p) \quad( x\in z+\sq_m).
\end{equation}
Since $v(\cdot, z+\sq_m,p)$ belongs to $l_p+H_0^1(z+\sq_m)$, it is easy to check that $v_m$ belongs to $l_p+H_0^1(\sq_n)$. This gives $\fint_{z+\sq_m}(\nab v_m)=p$. Using this and \eqref{lem:rq3}, we obtain
\begin{equation}
\fint_{z+\sq_m} \left(\nab v-p\right)=\fint_{z+\sq_m}\left(\nab v- \nab v_m\right) \label{pr:wc6}
\end{equation}
and
\begin{equation}
\fint_{z+\sq_m}\left(\aa\nab v-\bar{\aa}p\right)=\left(\aa(z+\sq_m)-\bar{\aa}\right)p+\fint_{z+\sq_m} \left( \aa(\nab v-\nab v_m) \right). \label{pr:wc7}
\end{equation}
From $\left| \left(\aa(z+\sq_m)-\bar{\aa}\right)p \right| \leq \Lam(\sq_n)+|\bar{\aa}|$, we see that
\begin{equation}
\left( \sum_{m=0}^{n-1} 3^{-(n-m)}\left( 3^{-(n-m)d}\sum_{z\in3^m\mathbb{Z}^d\cap \sq_{n}} \left| \left(\aa(z+\sq_m)-\bar{\aa}\right)p \right|^2\right)^{\frac{1}{2}}\right)^2 \leq \left( \Lam(\sq_n)+|\bar{\aa}| \right)\Om(n). \label{pr:wc8}
\end{equation}
Furthermore, Jensen's inequality yields
\begin{align}
3^{-(n-m)d}\sum_{z\in3^m\mathbb{Z}^d\cap \sq_{n}}\left| \fint_{z+\sq_m}\left(\nab v -\nab v_m \right)\right|^2 &\leq 3^{-(n-m)d}\sum_{z\in3^m\mathbb{Z}^d\cap \sq_{n}}\| \nab v-\nab v_m\|_{\ll^2(z+\sq_m)}^2 \\
&= \| \nab v-\nab v_m\|_{\ll^2(\sq_n)}^2 \label{pr:wc9}
\end{align}
and
\begin{equation}
3^{-(n-m)d}\sum_{z\in3^m\mathbb{Z}^d\cap \sq_{n}}\left|\fint_{z+\sq_m}\left(\aa(\nab v -\nab v_m )\right)\right|^2 \leq \Lam(\sq_n)^2\| \nab v-\nab v_m\|_{\ll^2(\sq_n)}^2. \label{pr:wc10}
\end{equation}
From \eqref{lem:qr2} and \eqref{lem:rq2}, it follows that
\begin{align}
&\sum_{m=0}^{n-1}3^{-(n-m)}\| \nab v-\nab v_m\|_{\ll^2(\sq_n)} \\
\leq& \sum_{m=0}^{n-1} 3^{-(n-m)}\left( \frac{2}{\lam(\sq_n)} 3^{-(n-m)d} \sum_{z\in3^m\mathbb{Z}^d\cap \sq_{n}}\left(\mu(z+\sq_m,p)-\mu(\sq_n,p) \right) \right)^{\frac{1}{2}} \\
\leq&\frac{1}{\lam(\sq_n)^{\frac{1}{2}}} \sum_{m=0}^{n-1} 3^{-(n-m)}\left( 3^{-(n-m)d} \sum_{z\in3^m\mathbb{Z}^d\cap \sq_{n}}\left| \aa(z+\sq_m)-\aa(\sq_n) \right| \right)^{\frac{1}{2}} \\
\leq&\frac{1}{\lam(\sq_n)^{\frac{1}{2}}} \sum_{m=0}^{n-1} 3^{-(n-m)}\left( \left| \aa(\sq_n)-\bar{\aa} \right|+3^{-(n-m)d} \sum_{z\in3^m\mathbb{Z}^d\cap \sq_{n}}\left| \aa(z+\sq_m)-\bar{\aa} \right| \right)^{\frac{1}{2}} \\
\leq&\frac{1}{\lam(\sq_n)^{\frac{1}{2}}} \left( \left| \aa(\sq_n)-\bar{\aa} \right|^{\frac{1}{2}}+\sum_{m=0}^{n-1} 3^{-(n-m)}\left( 3^{-(n-m)d} \sum_{z\in3^m\mathbb{Z}^d\cap \sq_{n}}\left| \aa(z+\sq_m)-\bar{\aa} \right| \right)^{\frac{1}{2}} \right) \\
=&\,C\left( \frac{\Om(n)}{\lam(\sq_n)} \right)^{\frac{1}{2}}.  \label{pr:wc11}
\end{align}
Combining \eqref{pr:wc6}, \eqref{pr:wc9} and \eqref{pr:wc11}, we obtain
\begin{equation}
\left( \sum_{m=0}^{n-1} 3^{-(n-m)}\left( 3^{-(n-m)d}\sum_{z\in3^m\mathbb{Z}^d\cap \sq_{n}} \left| \fint_{z+\sq_m}\left(\nab v-p\right) \right|^2\right)^{\frac{1}{2}} \right)^2 \leq C\frac{\Om(n)}{\lam(\sq_n)}. \label{pr:wc12}
\end{equation}
Thus, \eqref{pr:wc12}, \eqref{pr:wc3} and \eqref{pr:wc5} lead to \eqref{prop:ieqfl}. From \eqref{pr:wc7}, \eqref{pr:wc8}, \eqref{pr:wc10} and \eqref{pr:wc11}, we have
\begin{align}
&\left( \sum_{m=0}^{n-1} 3^{-(n-m)}\left( 3^{-(n-m)d}\sum_{z\in3^m\mathbb{Z}^d\cap \sq_{n}} \left| \fint_{z+\sq_m}\left(\aa\nab v-\bar{\aa}p\right) \right|^2\right)^{\frac{1}{2}} \right)^2 \\
\leq &\, C \left( \frac{\Lam(\sq_n)^2}{\lam(\sq_n)} + \Lam(\sq_n) + |\bar{\aa}| \right) \Om(n) \leq C\left( \frac{\Lam(\sq_n)^2}{\lam(\sq_n)} + |\bar{\aa}| \right) \Om(n). \label{pr:wc13}
\end{align}
Thus, \eqref{pr:wc13}, \eqref{pr:wc4} and \eqref{pr:wc5} lead to \eqref{prop:ieqde}, which completes the proof.
\end{proof}


\subsection{Error Estimate for the Dirichlet Problem} \label{sec:EEDP}
The following theorem gives an estimate of the error for the Dirichlet problem using $\Om(n)$. For a general domain $U$ and boundary condition $f$, we approximate the solution with a cutoff function and $v(\cdot,\sq_n,p)$, which results in the estimate of $\Om(n)$. 
\begin{thm} \label{thmthm:EEDP}
Suppose that there exist $a>0$ and $A < \infty$ satisfying 
\begin{equation}
a\textup{\textsf{Id}} \leq \bar{\aa} \leq A\textup{\textsf{Id}}. \label{ieq:aA}
\end{equation} 
Let $U \subseteq \sq_0$ be a bounded Lipschitz domain and $\del >0$. Then, there exist constants $b=b(d,a,A,U,\del)>0$ and $C=C(d,a,A,U,\del)< \infty$ such that the following holds: For every $\eps \in (0,1]$, $f \in W^{1,2+\del}(U)$, $n \in [-\log_3 \eps ,-\log_3 \eps +1)\cap \mathbb{N}$ and the unique solutions $u^\eps,u \in f+H_0^1(U)$ of the Dirichlet problems
\begin{alignat}{2}
-\nab \cdot \aa \left( \frac{\cdot}{\eps} \right) \nab u^\eps &=0 \quad(\text{in } U),  \qquad u^\eps&=f \quad( \text{on } \partial U), \\
-\nab \cdot \bar{\aa}  \nab u &=0 \quad(\text{in } U), \qquad u&=f \quad( \text{on } \partial U), \label{con:uequ}
\end{alignat}
in the distribution sense, we have, for every $r \in (0,1)$, that 
\begin{align}
\|u&-u^{\eps} \|_{L^2(U)}  \leq  C \frac{\Lam(\sq_n)+1}{\lam(\sq_n)} \left\| \nab f \right\|_{L^{2+\del}(U)}  \\ 
&\times\left( r^b+\frac{1}{r^{2+d/2}}  \left\{ \left( \frac{\Lam(\sq_n)^3+\Lam(\sq_n)}{\lam(\sq_n)}\right)^{\frac{1}{2}}\eps +\left( \left( \frac{\Lam(\sq_n)^2+1}{\lam(\sq_n)}\right)^{\frac{1}{2}}+1\right) \Om(n)^{\frac{1}{2}} \right\}\right). \label{thm:eedp}
\end{align}
\end{thm}
\begin{proof}
We write $U_r:=\left\{ x \in U: \text{dist}(x,\partial U)>r \right\}$. Let $\eta_r \in C_c^{\infty}(U)$ be a cutoff function satisfying, for every $k \in \{1,2,3 \}$,
\begin{equation}
0 \leq \eta_r \leq 1, \quad \eta_r=1 \;\text{in}\;U_{2r}, \quad \eta_r=0\; \text{in}\;U\setminus U_r, \quad |\nab^k \eta_r| \leq \frac{C}{r^k} \label{cutoff}
\end{equation}
and we set, for each $n \in \mathbb{N}$ and $p \in \mathbb{R}$, 
\begin{equation}
\phi_{n,p}(x):= v(x,\sq_n,p)-p\cdot x.
\end{equation}
Note that $-l_p+\phi_{n,p}$ is the solution of the Diriclet ploblem in $\sq_n$ with boundary condition $l_p$, and we have estimated $\phi_{n,p}$ in the previous section. The strategy of the proof is to approximate $u^{\eps}$ by the function
\begin{equation}
w^{\eps}(x):=u(x)+\eps \eta_r(x)\sum_{i=1}^{d} \partial_{x_i} u(x)\phi_{n,e_i}\left( \frac{x}{\eps}\right). \label{def:we}
\end{equation}
Most of what we show in this proof is that the error between $u$ and $u^{\eps}$ is estimated by the random variable
\begin{equation}
\Psi(\eps):=\sum_{i=1}^d \left( \eps \left\| \phi_{n,e_i} \left( \frac{\cdot}{\eps} \right)\right\|_{L^2(\eps \sq_n)}+\left\| \aa\left(\frac{\cdot}{\eps} \right) \left( e_i+ (\nab \phi_{n,e_i}) \left( \frac{\cdot}{\eps} \right)\right)-\bar{\aa}e_i\right\|_{H^{-1}(\eps\sq_n)} \right)^2.
\end{equation}

\textit{Step 1.} We note that the estimates obtained by $u$ being the solution to the differential equation with constant coefficients. First, we use mean value property. Since $u(\bar{\aa}^{\frac{1}{2}}\cdot)$ is harmonic in $\bar{\aa}^{-\frac{1}{2}}U$, it follows from mean value property that there exists constants $c=c(a,A) \in (0,1)$ and $C=C(d,a,A) < \infty$ such that, for every $z \in U_r$ and $k \in \{1,2,3\}$,
\begin{equation}
\| \nab ^k u \|_{L^{\infty}(B_{cr}(z))} \leq \frac{C}{r^{k-1}}\| \nab u \|_{\ll^2(B_r(z))} \leq \frac{C}{r^{k-1+d/2}} \| \nab u \|_{L^2(U)}. 
\end{equation}
Since $u$ is a weak solution of \eqref{con:uequ}, we get
\begin{equation}
\int_{U}\nab u \cdot \bar{\aa}\nab u = \inf_{v \in f+H_0^1(U)} \int_{U}\nab v \cdot \bar{\aa}\nab v   \leq  \int_{U}\nab f \cdot \bar{\aa}\nab f . \label{ieq:ufw}
\end{equation}
From \eqref{ieq:aA} and \eqref{ieq:ufw}, it follows that
\begin{equation}
\| \nab u \|_{L^2(U)}^2 \leq \frac{1}{a} \int_{U}\nab u \cdot \bar{\aa}\nab u \leq \frac{1}{a} \int_{U}\nab f \cdot \bar{\aa}\nab f \leq \frac{A}{a} \| \nab f \|_{L^2(U)}^2 .
\end{equation}
Combining the previous two inequalities shows that there exists a constant $C=C(k,d,a,A) < \infty$ such that
\begin{equation}
\| \nab^k u \|_{L^{\infty}(U_r)} \leq \frac{C}{r^{k-1+d/2}} \| \nab f \|_{L^2(U)}. \label{ieq:ahmu}
\end{equation}
Second, we use the Meyers estimate. From Lemma \ref{lemlem:meyers}, we get the existence of the constants $\del'=\del'(U,d,a,A) \in (0,\del)$ and $C=C(U,d,a,A)<\infty$ such that
\begin{equation}
\| \nab u \|_{L^{2+\del'}(U)} \leq C \| \nab f \|_{L^{2+\del'}(U)} \leq C \| \nab f \|_{L^{2+\del}(U)} .
\end{equation}
Since $U$ is a bounded Lipschitz domain, we obtain the estimate of $\nab u$ in the boundary layer by H\"{o}lder's inequality: for $b:=\frac{\del'}{4+2\del'}$,
\begin{equation}
\| \nab u \|_{L^2(U \setminus U_{2r})} \leq |U \setminus U_{2r}|^{b} \| \nab u \|_{L^{2+\del'}(U)} \leq Cr^b  \| \nab f \|_{L^{2+\del}(U)}. \label{ieq:blayer}
\end{equation}

\textit{Step 2.} In this step, we will get the estimate
\begin{equation}
\left\| \nab \cdot \left( \aa\left(\frac{\cdot}{\eps} \right)\nab w^{\eps} \right)\right\|_{H^{-1}(U)} \leq C \left( \Lam(\sq_n)+1 \right) \| \nab f \|_{L^{2+\del}(U)} \left(r^b+\frac{1}{r^{2+d/2}}\Psi(\eps)^{\frac{1}{2}} \right). \label{pr:eedpst2}
\end{equation}
Since $U\subseteq \sq_0 \subseteq \eps \sq_n$ and $\left(\nab v(\cdot,\sq_n,e_i) \right)\left(\frac{x}{\eps}\right)=e_i+(\nab\phi_{n,e_i})\left(\frac{x}{\eps}\right)$, it follows from \eqref{def:weak} that
\begin{align}
&\nab \cdot \left( \aa\left(\frac{\cdot}{\eps}\right) \eta_r (\partial_{x_i} u)\left( e_i+(\nab \phi_{e_i})\left(\frac{\cdot}{\eps} \right) \right)\right) \\
=& \, \nab \left( \eta_r \partial_{x_i}u\right)\cdot \left( \aa\left( \frac{\cdot}{\eps}\right)\left( e_i+(\nab \phi_{n,e_i})\left( \frac{\cdot}{\eps}\right)\right) \right) +\eta_r (\partial_{x_i}u) \nab \cdot \left( \aa\left( \frac{\cdot}{\eps}\right)\left( e_i+(\nab \phi_{n,e_i})\left( \frac{\cdot}{\eps}\right)\right) \right) \\
=& \, \nab \left( \eta_r \partial_{x_i}u\right)\cdot \left( \aa\left( \frac{\cdot}{\eps}\right)\left( e_i+(\nab \phi_{n,e_i})\left( \frac{\cdot}{\eps}\right)\right) \right) . \label{eq:sasa1}
\end{align}
From \eqref{def:we}, we have
\begin{align}
&\nab w^{\eps}=\nab u+ \eta_r \sum_{i=1}^{d} (\partial_{x_i} u) ( \nab \phi_{n,e_i})\left( \frac{\cdot}{\eps}\right) + \eps \sum_{i=1}^{d} \nab \left( \eta_r \partial_{x_i} u \right) \phi_{n,e_i} \left( \frac{\cdot}{\eps}\right) \\
=&\sum_{i=1}^d \left( \eta_r (\partial_{x_i} u) \left( e_i+(\nab \phi_{n,e_i})\left( \frac{\cdot}{\eps}\right)\right) + \left( 1-\eta_r \right)(\partial_{x_i} u)e_i + \eps \, \nab \left( \eta_r \partial_{x_i} u \right) \phi_{n,e_i} \left( \frac{\cdot}{\eps}\right) \right). \label{eq:sasa2}
\end{align}
Combining \eqref{eq:sasa1} and \eqref{eq:sasa2}, we obtain
\begin{align}
\nab \cdot \left( \aa\left(\frac{\cdot}{\eps}\right)\nab w^{\eps} \right)=&\sum_{i=1}^d \nab \left( \eta_r \partial_{x_i}u\right)\cdot \aa\left( \frac{\cdot}{\eps}\right)\left( e_i+(\nab \phi_{n,e_i})\left( \frac{\cdot}{\eps}\right)\right) \\
&+\nab \cdot \left( \aa\left(\frac{\cdot}{\eps}\right) \left( (1-\eta_r)\nab u + \eps \sum_{i=1}^{d} \phi_{n,e_i} \left( \frac{\cdot}{\eps}\right)\nab(\eta_r \partial_{x_i}u)\right)\right). \label{ieq:3hhhh}
\end{align}
Then, \eqref{con:uequ} leads to 
\begin{equation}
\sum_{i=1}^{d} \nab(\eta_r \partial_{x_i}u)\cdot \bar{\aa}e_i=\nab \cdot (\eta_r \bar{\aa}\nab u) = -\nab \cdot \left( (1-\eta_r)\bar{\aa}\nab u\right). \label{ieq:3hhhi}
\end{equation}
From \eqref{ieq:3hhhh} and \eqref{ieq:3hhhi}, we have
\begin{align}
\nab \cdot \left( \aa\left(\frac{\cdot}{\eps}\right)\nab w^{\eps} \right)=&\sum_{i=1}^d \nab \left( \eta_r \partial_{x_i}u\right)\cdot \left(\aa\left( \frac{\cdot}{\eps}\right)\left( e_i+(\nab \phi_{n,e_i})\left( \frac{\cdot}{\eps}\right)\right)-\bar{\aa}e_i \right) \\
&+\nab \cdot \left( (1-\eta_r)\left(\aa\left(\frac{\cdot}{\eps}\right)-\bar{\aa} \right)\nab u\right)+ \nab \cdot \left( \eps \sum_{i=1}^d \phi_{e_i}\left( \frac{\cdot}{\eps}\right)\aa\left( \frac{\cdot}{\eps}\right)\nab(\eta_r \partial_{x_i}u)\right).
\end{align}
It follows that, for every $F:U \to \mathbb{R}^d$ whose entries are belong to $H^1(U)$,
\begin{align}
\| \nab \cdot F \|_{H^{-1}(U)} &= \sup \left\{ \int_{U} \left( \nab \cdot F\right) v : v \in H_0^1(U),\|v\|_{H^1(U)}\leq1\right\} \\
&= \sup \left\{ \int_{U} F\cdot \nab v : v \in H_0^1(U),\|v\|_{H^1(U)}\leq1\right\}  \leq C \| F\|_{L^2(U)} .
\end{align}
Combining the previous two inequalities, we obtain
\begin{align}
&\| \nab \cdot \left( \aa\left(\frac{\cdot}{\eps}\right)\nab w^{\eps} \right) \|_{H^{-1}(U)} \\
\leq& \sum_{i=1}^d \left\| \nab \left( \eta_r \partial_{x_i}u\right) \right\|_{W^{1,\infty}(U)} \left\| \aa\left( \frac{\cdot}{\eps}\right)\left( e_i+(\nab \phi_{n,e_i})\left( \frac{\cdot}{\eps}\right)\right)-\bar{\aa}e_i \right\|_{H^{-1}(\eps\sq_n)} \\
&+C\left\| (1-\eta_r)\left(\aa\left(\frac{\cdot}{\eps}\right)-\bar{\aa} \right)\nab u \right\|_{L^2(U)} +C\sum_{i=1}^{d} \left\|  \eps \phi_{n,e_i}\left( \frac{\cdot}{\eps}\right)\aa\left( \frac{\cdot}{\eps}\right)\nab(\eta_r \partial_{x_i}u)\right\|_{L^2(U)}.
\end{align}
For the second term on the right-hand side, we use \eqref{ieq:blayer} and the fact that $1- \eta_r =0$ in $U_{2r}$ to obtain
\begin{align}
\left\| (1-\eta_r)\left(\aa\left(\frac{\cdot}{\eps}\right)-\bar{\aa} \right)\nab u \right\|_{L^2(U)} &\leq \left( \Lam(\sq_n)+| \bar{\aa} | \right) \left\| (1-\eta_r)\nab u \right\|_{L^2(U)}  \\
&\leq C\left( \Lam(\sq_n)+1 \right) r^b  \| \nab f \|_{L^{2+\del}(U)} .
\end{align}
On the other hand, \eqref{cutoff} and \eqref{ieq:ahmu} imply
\begin{equation}
\left\| \nab (\eta_r \partial_{x_i} u)\right\|_{W^{1,\infty}(U)} \leq \frac{C}{r^{2+d/2}} \left\| \nab f\right\|_{L^2(U)}.
\end{equation}
Combining the previous three inequalities and the definition of $\Psi(\eps)$ yields \eqref{pr:eedpst2}.

\textit{Step 3.}
In this step, we claim that
\begin{equation}
\|\nab u^{\eps}- \nab w^{\eps} \|_{\widehat{H}^{-1}(U)} \leq C \frac{\Lam(\sq_n)+1}{\lam(\sq_n)}  \| \nab f \|_{L^{2+\del}(U)} \left( r^b+ \frac{1}{r^{2+d/2}} \Psi(\eps)^{\frac{1}{2}} \right). \label{pr:eedpst3}
\end{equation}
Since $u^{\eps}-w^{\eps} \in H_0^1(U)$, it follows that
\begin{equation}
\left| \int_{U} \nab (u^{\eps}-w^{\eps}) \cdot \aa\left( \frac{\cdot}{\eps} \right)\nab w^{\eps} \right| \leq \| u^{\eps}-w^{\eps} \|_{H^1(U)} \left\| \nab \cdot \left( \aa\left(\frac{\cdot}{\eps} \right)\nab w^{\eps} \right)\right\|_{H^{-1}(U)}. \label{ieq:st3a}
\end{equation}
Testing the equation for $u^{\eps}$ with $u^{\eps}-w^{\eps} \in H_0^1(U)$ leads to 
\begin{equation}
\int_{U} \nab (u^{\eps}-w^{\eps}) \cdot \aa\left( \frac{\cdot}{\eps} \right)\nab u^{\eps} =0. \label{ieq:st3b}
\end{equation}
Combining \eqref{ieq:st3a} and \eqref{ieq:st3b} and using Poincar\'{e} inequality, we have
\begin{align}
\|\nab u^{\eps}- \nab w^{\eps} \|_{L^2(U)}^2 &\leq \frac{1}{\lam(\sq_n)} \int_{U} \nab (u^{\eps}-w^{\eps}) \cdot \aa\left( \frac{\cdot}{\eps} \right)\nab (u^{\eps}-w^{\eps}) \\
&\leq \frac{1}{\lam(\sq_n)} \| u^{\eps}-w^{\eps} \|_{H^1(U)} \left\| \nab \cdot \left( \aa\left(\frac{\cdot}{\eps} \right)\nab w^{\eps} \right)\right\|_{H^{-1}(U)} \\
&\leq \frac{C}{\lam(\sq_n)} \| \nab u^{\eps}- \nab w^{\eps} \|_{L^2(U)} \left\| \nab \cdot \left( \aa\left(\frac{\cdot}{\eps} \right)\nab w^{\eps} \right)\right\|_{H^{-1}(U)},
\end{align}
and thus
\begin{equation}
\| \nab u^{\eps}- \nab w^{\eps} \|_{L^2(U)} \leq \frac{C}{\lam(\sq_n)} \left\| \nab \cdot \left( \aa\left(\frac{\cdot}{\eps} \right)\nab w^{\eps} \right)\right\|_{H^{-1}(U)}.
\end{equation}
The fact that $u_{\eps}-w_{\eps} \in H_0^1(U)$ and Poincar\'{e} inequality lead to
\begin{equation}
\|\nab u^{\eps}- \nab w^{\eps} \|_{\widehat{H}^{-1}(U)} \leq  C\| u^{\eps}- w^{\eps} \|_{L^2(U)} \leq C \|\nab u^{\eps}- \nab w^{\eps} \|_{L^2(U)}.
\end{equation}
Hence, by using \eqref{pr:eedpst2}, \eqref{pr:eedpst3} is proved.

\textit{Step 4.} We estimate the $L^2$-norm of difference between $u^{\eps}$ and $ u$ by using $\Psi(\eps)$. The claim is that
\begin{equation}
\left\| u-u^{\eps} \right\|_{L^2(U)} \leq C \frac{\Lam(\sq_n)+1}{\lam(\sq_n)} \| \nab f \|_{L^{2+\del}(U)} \left( r^b+ \frac{1}{r^{2+d/2}} \Psi(\eps)^{\frac{1}{2}} \right). \label{pr:eedpst4}
\end{equation}
Since
\begin{equation}
w^{\eps}-u=\eps\eta_r \sum_{i=1}^d \partial_{x_i}u \phi_{n,e_i} \left( \frac{\cdot}{\eps} \right)
\end{equation}
is belong to $H_0^1(U)$, we have 
\begin{equation}
\| \nab w^{\eps}-\nab u \|_{\widehat{H}^{-1}(U)} \leq C \left\| \eps \eta_r \sum_{i=1}^d \partial_{x_i}u\phi_{n,e_i} \left( \frac{\cdot}{\eps}\right)\right\|_{L^2(U)} \leq C \|\nab u \|_{L^{\infty}(U_r)}\sum_{i=1}^d \eps \left\| \phi_{n,e_i} \left( \frac{\cdot}{\eps} \right)\right\|_{L^2(\eps\sq_n)}.
\end{equation} 
From this and \eqref{ieq:ahmu}, we have
\begin{equation}
\| \nab w^{\eps}-\nab u \|_{\widehat{H}^{-1}(U)} \leq \frac{C}{r^{d/2}} \| \nab f \|_{L^{2+\del}(U)} \Psi(\eps)^{\frac{1}{2}}.
\end{equation}
From this and \eqref{pr:eedpst3}, we obtain  
\begin{equation}
\| \nab u^{\eps}-\nab u \|_{\widehat{H}^{-1}(U)} \leq C \frac{\Lam(\sq_n)+1}{\lam(\sq_n)} \| \nab f \|_{L^{2+\del}(U)} \left( r^b+ \frac{1}{r^{2+d/2}} \Psi(\eps)^{\frac{1}{2}} \right).
\end{equation}
Since $u^{\eps}-u\in H_0^1(U)$ can be extended to $\sq_0$ by setting it to be $0$ on $\sq_0 \setminus U$, we apply Lemma \ref{lemlem:mpi} to obtain $\| u^{\eps}- u \|_{L^2(U)} \leq C \| \nab u^{\eps}-\nab u \|_{\widehat{H}^{-1}(U)}$, and we thus get \eqref{pr:eedpst4}.

\textit{Step 5.} To conclude the proof, it remains to estimate $\Psi(\eps)$ by $\Om(n)$. First we note the $\widehat{H}^{-1}$ norm of rescaling: for every $f \in L^2(\sq_n)$,
\begin{equation}
\left\| f \left( \frac{\cdot}{\eps} \right) \right\|_{\underline{H}^{-1}(\eps\sq_n)}=\eps \| f \|_{\underline{H}^{-1}(\sq_n)}.
\end{equation}
We also note that $| \eps\sq_n | < 3^d$ holds. Recall that $\phi_{n,e_i}=v(\cdot,\sq_n,e_i)-l_{e_i}$ belongs to $H_0^1(\sq_n)$. By applying Lemma \ref{lemlem:mpi} and Proposition \ref{propprop:ieqfl}, we obtain
\begin{align}
\Psi(\eps) \leq&\, C \sum_{i=1}^d \left( \eps^2 \left\| \phi_{n,e_i} \left( \frac{\cdot}{\eps} \right)\right\|_{\underline{L}^2(\eps \sq_n)}^2+\left\| \aa\left(\frac{\cdot}{\eps} \right) \left( e_i+ (\nab \phi_{n,e_i}) \left( \frac{\cdot}{\eps} \right)\right)-\bar{\aa}e_i\right\|_{\underline{H}^{-1}(\eps\sq_n)}^2 \right)  \\
\leq & \, C \eps^2 \sum_{i=1}^d \left( \left\| \phi_{n,e_i} \right\|_{\underline{L}^2(\sq_n)}^2+\left\| \aa( e_i+ \nab \phi_{n,e_i})-\bar{\aa}e_i\right\|_{\underline{H}^{-1}(\sq_n)}^2 \right)  \\
\leq&\, C \eps^2 \sum_{i=1}^d \left( \left\| \nab \phi_{n,e_i} \right\|_{\hh^{-1}(\sq_n)}^2+\left\| \aa( e_i+ \nab \phi_{n,e_i})-\bar{\aa}e_i\right\|_{\hh^{-1}(\sq_n)}^2 \right) \\
\leq&\, C \eps^2  \left\{ \left( \frac{\Lam(\sq_n)^3+\Lam(\sq_n)}{\lam(\sq_n)}+|\bar{\aa}|^2\right)+3^{2n}\left( \frac{\Lam(\sq_n)^2+1}{\lam(\sq_n)}+|\bar{\aa}|\right)\Om(n) \right\} \\
\leq&\, C \left( \frac{\Lam(\sq_n)^3+\Lam(\sq_n)}{\lam(\sq_n)} \right) \eps^2 + C \left( \frac{\Lam(\sq_n)^2+1}{\lam(\sq_n)}+1\right)\Om(n).
\end{align}
Combining this inequality and \eqref{pr:eedpst4} gives \eqref{thm:eedp}.
\end{proof}

\subsection{Proof of Theorem \ref{thm:king}} \label{proofking}
In this section, we prove Theorem \ref{thm:king} by combining the results in Section \ref{sec:QCSQ} and Section \ref{sec:EEDP}. Theorem \ref{thmthm:EEDP} indicates that the effect of the ellipticity of the coefficients on the error is almost polynomial order. From Lemma \ref{propprop:epmax}, $\Lam(\sq_n)$ and $1/\lam(\sq_n)$ grow almost as much as polynomials, so they do not interfere with the rate $\exp\left( -cn^{1-3\alpha}\right)$ of convergence of $\aa(\sq_n)$, which is obtained in theorem \ref{thm:coa}.
\begin{proof}[Proof of Theorem \ref{thm:king}]
\textit{Step 1.} In this step, we prove the existence of suppressive sequences. We set 
\begin{equation}
\beta':=\frac{1}{\beta}+\frac{1}{2}\left( \alpha -\frac{1}{\beta}-\frac{1}{\gamma}\right), \quad \gamma':=\frac{1}{\gamma}+\frac{1}{2}\left( \alpha -\frac{1}{\beta}-\frac{1}{\gamma}\right),
\end{equation}
and $\del_n:=(n+1)^{-\gamma'},M_n:=(n+1)^{\beta'}$. We show that there exists a constant $L=L(M,\beta,\gamma,\alpha)< \infty$ such that $(\{\del_n\},\{M_n\})$ is suppressive.  By Chebyshev's inequality, it follows that
\begin{equation}
\prb\left( \Lam(\sq_0) \geq M_n\right) = \prb \left( \exp \left( \Lam(\sq_0)^{\beta}\right) \geq \exp \left( (n+1)^{\beta\beta'}\right)\right) \leq \exp \left( -(n+1)^{\beta\beta'}\right) \ep \left[ \exp \left( \Lam(\sq_0)^{\beta}\right) \right].
\end{equation}
By using stationarity and $\beta\beta'>1$, this implies the estimate
\begin{align}
\prb\left( \Lam(\sq_n) \geq M_n \right) &=\prb \left( \bigcup_{z \in  \mathbb{Z}^d\cap \sq_{n}} \{ \Lam(z+\sq_0)\geq M_n\} \right) \\
&\leq 3^{nd} \,\prb \left( \Lam(\sq_0)\geq M_n \right) \leq M 3^{nd} \exp \left( -(n+1)^{\beta\beta'}\right) \leq Ce^{-4n}.
\end{align}
On the other hand, since 
\begin{equation}
\ep\left[ \exp \left(\Lam(\sq_0) \right) \right] \leq \ep\left[ \exp \left(\Lam(\sq_0)^{\beta} \right) :\Lam(\sq_0) \geq 1\right]+\ep\left[ \exp \left(\Lam(\sq_0) \right) :\Lam(\sq_0) \leq 1\right] \leq M+e,
\end{equation}
it follows from Proposition \ref{propprop:epmax} that, for every $q\geq1$, there exists $C=C(q,M)<\infty$ such that
\begin{equation}
\ep \left[ \Lam(\sq_n)^q \right] =\ep \left[ \sup_{z \in \mathbb{Z}^d \cap \sq_n} \Lam(z+\sq_0)^q \right] \leq C \left((q-1)^q+\left( \log\left(3^{nd}(M+e)\right)\right)^q\right) \leq C+Cn^q. \label{pr:epLamn}
\end{equation}
In the same manner we can see that
\begin{equation}
\prb\left( \lam(\sq_n) \leq \del_n \right) \leq Ce^{-4n},
\end{equation}
and
\begin{equation}
\ep \left[ \lam(\sq_n)^{-q} \right]\leq C+Cn^q. \label{pr:eplamn}
\end{equation}
Using Cauchy-Schwarz inequality and the previous estimates, we obtain
\begin{align}
&\ep \left[ \lam\left(\square_n\right)^{-3}+\Lam\left(\sq_n\right)^3:\left\{ \lam\left(\sq_n\right)\leq \del_n\right\} \cup \left\{ \Lam\left(\sq_n\right)\geq M_n\right\} \right] \\
\leq & C \,\ep \left[ \lam\left(\square_n\right)^{-6}+\Lam\left(\sq_n\right)^6 \right]^{\frac{1}{2}} \prb \left( \left\{ \lam\left(\sq_n\right)\leq \del_n\right\} \cup \left\{ \Lam\left(\sq_n\right)\geq M_n\right\} \right)^{\frac{1}{2}} \\
\leq& C(1+n^3)e^{-2n} \leq Ce^{-n},
\end{align}
which gives our claim.

\textit{Step 2.} Next we get the estimate of the expectation of $\Om(n)^2$. By Cauchy-Schwarz inequality,
\begin{align}
\Om(n)^2 &\leq \left( \left( \sum_{m=0}^n 3^{-(n-m)} \right)^{\frac{1}{2}} \left( \sum_{m=0}^n 3^{-(n-m)} 3^{-(n-m)d} \sum_{z \in 3^m\mathbb{Z}^d \cap \sq_n} | \aa(z+\sq_m)-\bar{\aa}|  \right)^{\frac{1}{2}} \right)^4 \\
&\leq C \left( \sum_{m=0}^n 3^{-(n-m)} 3^{-(n-m)d} \sum_{z \in 3^m\mathbb{Z}^d \cap \sq_n} | \aa(z+\sq_m)-\bar{\aa}| \right)^2 \\
&\leq  C \left( \sum_{m=0}^{n} 3^{-(n-m)} \right) \left( \sum_{m=0}^{n} 3^{-(n-m)}3^{-2(n-m)d} \left( \sum_{z \in 3^m\mathbb{Z}^d \cap \sq_n} | \aa(z+\sq_m)-\bar{\aa}| \right)^2\right) \\
&\leq C \sum_{m=0}^{n} 3^{-(n-m)} 3^{-(n-m)d} \sum_{z \in 3^m\mathbb{Z}^d \cap \sq_n} | \aa(z+\sq_m)-\bar{\aa}|^2. \label{holder}
\end{align}
For $x \in \mathbb{R}$, we denote by $\lfloor x \rfloor$ the greatest integer less than or equal to $x$. From the existence of suppressive sequences, taking expectation and using stationarity and Theorem \ref{thm:coa} yield
\begin{align}
\ep \left[ \Om(n)^2 \right] &\leq C \sum_{m=0}^n 3^{-(n-m)} \exp \left( -cm^{1-3\alpha} \right) \\
&= \sum_{m=0}^{\left\lfloor \frac{n}{2} \right\rfloor} 3^{-(n-m)} \exp \left( -cm^{1-3\alpha} \right) + \sum_{m=\left\lfloor \frac{n}{2} \right\rfloor+1}^{n} 3^{-(n-m)} \exp \left( -cm^{1-3\alpha} \right) \\
&\leq C3^{-\frac{n}{2}}+C\exp\left( -c\left\lfloor \frac{n}{2}+1\right\rfloor^{1-3\alpha}\right) \leq C \exp \left(-cn^{1-3\alpha}\right). \label{pr:epom}
\end{align}

\textit{Step 3.} We conclude the proof. Fix $p \in (0,4)$. For $\eps \in (0,1)$, we set $n \in [-\log_3 \eps ,-\log_3 \eps +1)\cap \mathbb{N}$ and we write
\begin{equation}
\Phi(\eps):=\left( \frac{\Lam(\sq_n)^3+\Lam(\sq_n)}{\lam(\sq_n)}\right)^{\frac{1}{2}}\eps +\left( \left( \frac{\Lam(\sq_n)^2+1}{\lam(\sq_n)}\right)^{\frac{1}{2}}+1\right) \Om(n)^{\frac{1}{2}}.
\end{equation}
Using H\"{o}lder's inequality, \eqref{pr:epLamn}, \eqref{pr:eplamn}, \eqref{pr:epom} and $-\log_3\eps\leq n$ leads to
\begin{align}
&\ep\left[ \Phi(\eps)^{\frac{p+4}{2}}\right] \\
\leq&\, C\,\ep \left[ \left( \frac{\Lam(\sq_n)^3+\Lam(\sq_n)}{\lam(\sq_n)}\right)^{\frac{p+4}{4}}\right]\eps^{\frac{p+4}{2}}+C\,\ep\left[ \left( \left(\frac{\Lam(\sq_n)^2+1}{\lam(\sq_n)}\right)^{\frac{1}{2}}+1\right)^{\frac{4(p+4)}{4-p}}\right]^{\frac{4-p}{8}} \ep\left[ \Om(n)^2\right]^{\frac{p+4}{8}} \\
\leq&\, C\left(n^{p+4}+1\right) \eps^{\frac{p+4}{2}}+C\left(n^{\frac{3(p+4)}{4}}+1\right) \exp \left(-cn^{1-3\alpha}\right) \\
\leq&\, C \exp \left(-c\left( -\log \eps \right)^{1-3\alpha} \right). \label{pr:epPhi}
\end{align}
In using Theorem \ref{thmthm:EEDP}, two things should be mentioned. First, since there exist suppressive sequences, Proposition \ref{propprop:baa} and \eqref{supep0} give the existence of the constants $a=a(M,\beta,\gamma,\alpha)>0$, $A=A(M,\beta,\gamma,\alpha)<\infty$\ satisfying $a\textup{\textsf{Id}} \leq \bar{\aa} \leq A\textup{\textsf{Id}}$. Second, for $b,c>0$ and $D>0$, the following equality holds:
\begin{equation}
\inf_{r \in (0,1)} \left(r^b+\frac{1}{r^c}D \right)=
\begin{dcases}
\left( \bigg( \frac{c}{b} \bigg)^{\frac{b}{b+c}} +\bigg( \frac{b}{c} \bigg)^{\frac{c}{b+c}} \right) D^{\frac{b}{b+c}} & \text{if }D\leq \frac{b}{c} \text{,}\\
1+D & \text{if }D\geq \frac{b}{c} \text{.}
\end{dcases}
\end{equation}
Let $c:=2+d/2$. Applying Theorem \ref{thmthm:EEDP} and using H\"{o}lder's inequality, \eqref{pr:epLamn}, \eqref{pr:eplamn} and the above inequality, we obtain
\begin{align}
&\ep \left[ \left\|u-u^{\eps} \right\|_{L^2(U)}^p \right] \\
\leq& \,C \left\| \nab f \right\|_{L^{2+\del}(U)}^p \ep \left[ \left(\frac{\Lam(\sq_n)+1}{\lam(\sq_n)}\right)^\frac{p(3p+4)}{4-p} \right]^{\frac{4-p}{3p+4}} \ep \left[ \inf_{r \in (0,1)} \left( r^b+\frac{1}{r^c}  \Phi(\eps) \right)^{\frac{3p+4}{4}} \right]^{\frac{4p}{3p+4}} \\
\leq& \,C\left\| \nab f \right\|_{L^{2+\del}(U)}^p\left( n^{2p}+1\right) \left( \ep \left[ \Phi(\eps)^{\frac{3p+4}{4}\frac{b}{b+c}} \right] +\ep \left[ \left( 1+\Phi(\eps) \right)^{\frac{3p+4}{4}}:\Phi(\eps) \geq \frac{b}{c} \right]\right)^{\frac{4p}{3p+4}}.
\end{align}
Using H\"{o}lder's inequality, Chebyshev's inequality and \eqref{pr:epPhi}, we obtain
\begin{equation}
\ep \left[ \Phi(\eps)^{\frac{3p+4}{4}\frac{b}{b+c}} \right] \leq \ep \left[ \Phi(\eps)^{\frac{p+4}{2}} \right]^{\frac{3p+4}{2p+8}\frac{b}{b+c}} \leq C\exp \left(-c\left( -\log \eps \right)^{1-3\alpha} \right)
\end{equation}
and
\begin{align}
\ep \left[ \left( 1+\Phi(\eps) \right)^{\frac{3p+4}{4}}:\Phi(\eps) \geq \frac{b}{c} \right] &\leq \prb \left( \Phi(\eps) \geq \frac{b}{c} \right)^{\frac{4-p}{2p+8}} \ep \left[ \left( 1+\Phi(\eps) \right)^{\frac{p+4}{2}}\right]^{\frac{3p+4}{2p+8}} \\
&\leq C\,\ep\left[ \Phi(\eps)^{\frac{p+4}{2}}\right]^{\frac{4-p}{2p+8}}\left( 1+\ep\left[ \Phi(\eps)^{\frac{p+4}{2}}\right]\right)^{\frac{3p+4}{2p+8}} \\
&\leq C\exp \left(-c\left( -\log \eps \right)^{1-3\alpha} \right).
\end{align}
Combining the previous three displays gives \eqref{thm:king2}, and the proof is complete.
\end{proof}


\section{Appendix}
In this section, we provide some estimates needed for our proof.

The following lemma gives us an upper estimate for the $L^p$-moment of the maximum of random variables.

\begin{prop} \label{propprop:epmax} 
\textup{(cf.\cite[Proposition 5.2.]{MR3772806})} Let $X_1,X_2,\dots,X_n$ are random variables such that $\ep [ e^{|X_i|} ] < \infty$ for $i \in \{1,2,\dots,n\}$. Then, there exists a constant $C=C\left(p\right)<\infty$ such that 
\begin{equation}
\ep\left[ \max_{i=1,2,\dots,n} \left|X_i\right|^p \right] \leq C\left\{ (p-1)^p+\left[ \log \left( \sum_{i=1}^{n} \ep \left[ e^{|X_i|} \right] \right) \right]^p \right\}. \label{rvmax}
\end{equation}
for every $n \in \mathbb{N}$ and every $p \in [1,\infty)$.
\end{prop}

\begin{proof}
Define a function $f$ on $[0,\infty)$ by
\begin{equation}
f(x):=\exp \left(x^{\frac{1}{p}} \right)
\end{equation}
and $x_0:=(p-1)^p$. It is easily seen that $f$ is strictly increasing on $[0,\infty)$ and convex on $[x_0,\infty)$. Set a ramdom variable $Y$ by
\begin{equation}
Y:=\max_{i=1,2,\dots,n} \left|X_i \right|^p \; \vee x_0.
\end{equation}
By applying Jensen's inequality to $Y$, we have
\begin{align}
&\ep\left[\max_{i=1,2,\cdots,n} \left| X_i \right|^p \right] \\
\leq& \ep \left[ Y \right] \\
\leq& f^{-1} \left( \ep \left[ f \left(Y \right) \right] \right) \\
=& f^{-1} \left( \ep \left[ \exp\left(\max_{i=1,2,\dots,n} |X_i|\right):Y \geq x_0 \right] +\exp\left(x_0^{1/p} \right) \prb \left( Y\leq x_0 \right) \right)  \\
\leq& f^{-1} \left( \ep \left[ \exp\left(\max_{i=1,2,\dots,n} |X_i|\right):Y \geq x_0 \right] + e^{p-1} \ep \left[ \exp\left(\max_{i=1,2,\dots,n} |X_i|\right):Y \leq x_0 \right] \right) \\
\leq& f^{-1} \left( e^{p-1} \ep \left[ \max_{i=1,2,\dots,n} e^{|X_i|} \right] \right) \\
\leq& \left(p-1 + \log\left( \sum_{i=1}^{n} \ep \left[ e^{|X_i|} \right] \right) \right)^p \\
\leq&2^{p-1} \left\{ \left( p-1 \right)^p + \left[ \log \left( \sum_{i=1}^{n} \ep \left[ e^{|X_i|} \right] \right) \right]^p \right\}.
\end{align}
Thus, \eqref{rvmax} is proved.
\end{proof}

We give the statement of the interior Caccioppoli inequality. For each $r \in \mathbb{R}_+$, we define $r\square:=( -r,r )^d$.
\begin{lem}[Interior Caccioppoli inequality] \label{lem:cacci}
Let $r>0$, $0<\lam\leq\Lam$ and let $\aa$ be a measurable map from $3r\square$ to the set of positive symmetric matrices with eigenvalues belonging to $[\lam,\Lam ]$. Suppose that $u \in H^1(3r\square)$ satisfies
\begin{equation}
-\nab \cdot \left( \aa(x) \nab u \right)=0 \quad \text{in} \; \; 3r\square.  \label{asm:ci}
\end{equation}
Then, there exists a constant $C=C(d)<\infty$ such that
\begin{equation}
\| \nab u\|_{\ll^2(r\square)}\leq C\,\frac{\Lam}{\lam}\,\frac{1}{r}\|u-(u)_{3r\square}\|_{\ll^2(3r\square)}. \label{lem:ci}
\end{equation}
\end{lem}
\begin{proof}
By replaceing $u-(u)_{3r\square}$ with $u$, we may suppose $(u)_{3r\square}=0$. Let $\phi \in C_c^{\infty}(3r\square)$ be a cutoff function satisfying
\begin{equation}
0\leq \phi \leq 1, \quad \phi=1\quad \text{in}\;\;r\square,  \quad \left| \nab \phi \right| \leq \frac{2}{r}, \label{def:phi}
\end{equation}
and test the equation \eqref{asm:ci} with $\phi^2u\in H_0^1(3r\square)$ to get
\begin{equation}
\fint_{3r\square}\phi^2\nab u \cdot \aa \nab u=-\fint_{3r\square} 2 \phi u \nab \phi \cdot \aa \nab u .
\end{equation}
From Young's inequality, the upper bound $\aa \leq \Lam\textup{\textsf{Id}}$ and \eqref{def:phi}, we have
\begin{align}
\left| 2\phi u \nab \phi \cdot \aa \nab u \right| &\leq \frac{2\Lam^2}{\lam}\left| u\nab \phi \right|^2+\frac{\lam}{2\Lam^2} \left| \phi \cdot \aa \nab u \right|^2 \\
&\leq \frac{8\Lam^2}{\lam} \, \frac{1}{r^2} \left| u \right|^2+\frac{\lam}{2}\left|\phi\right|^2\left|\nab u \right|^2.
\end{align}
From the lower bound $\lam\textup{\textsf{Id}}\leq \aa$, we have
\begin{equation}
\fint_{3r\square} \phi^2 \nab u \cdot \aa \nab u \geq \lam \fint_{3r\square} \phi^2 \left| \nab u \right|^2.
\end{equation}
By the three previous displays, it follows that 
\begin{equation}
\fint_{3r\square} \phi^2 \left| \nab u \right|^2 \leq C\,\frac{\Lam^2}{\lam^2}\,\frac{1}{r^2}\fint_{3r\square} \left| u\right|^2.
\end{equation}
Since $\| \phi \nab u \|_{\ll^2(3r\square)} \geq C(d) \|\nab u \|_{\ll^2(r\square)}$ by \eqref{def:phi}, the lemma follows. 
\end{proof}

The following lemma is an adapted version of Poincar\'{e} inequality in the domain $\sq_n$. It includes spatial averages on all triadic subcubes of $\sq_n$.
\begin{lem}[Multiscale Poincar\'{e} inequality]\label{lemlem:mpi}
There exists a constant $C(d)<\infty$ such that, for every $n,m\in\mathbb{N}$ with $m \leq n$ and every $u \in L^2(\sq_n)$,
\begin{equation}
\| u \|_{\hh^{-1}(\sq_n)} \leq C\,\| u \|_{\ll^{2}(\sq_n)} +C \sum_{m=0}^{n-1}3^m \left( 3^{-(n-m)d} \sum_{ y \in 3^m\mathbb{Z}^d\cap \sq_{n}} \left| \fint_{y+\sq_m}  u  \right|^2 \right)^{\frac{1}{2}}.
\end{equation}
Also, for every $v \in H^1(\sq_n)$,
\begin{align}
\|v-(v)_{\sq_n}\|_{\ll^2(\sq_n)} &\leq C\,\|\nab v \|_{\hh^{-1}(\sq_n)} \\
&\leq C\,\|\nab v \|_{\ll^{2}(\sq_n)} +C \sum_{m=0}^{n-1}3^m \left( 3^{-(n-m)d} \sum_{ y \in 3^m\mathbb{Z}^d\cap \sq_{n}} \left| \fint_{y+\sq_m} \nab v  \right|^2 \right)^{\frac{1}{2}}
\end{align}
and, for every $w \in H_0^1(\sq_n)$,
\begin{align}
\|w\|_{\ll^2(\sq_n)}  &\leq C\,\|\nab w \|_{\hh^{-1}(\sq_n)} \\
&\leq C\,\|\nab w \|_{\ll^{2}(\sq_n)}+C\sum_{m=0}^{n-1}3^m \left( 3^{-(n-m)d} \sum_{ y \in 3^m\mathbb{Z}^d\cap \sq_{n}} \left| \fint_{y+\sq_m}\nab w \right|^2 \right)^{\frac{1}{2}}.
\end{align}
\end{lem}
\begin{proof}
See \cite[Proposition 1.12. and Lemma 1.13.] {MR3932093}.
\end{proof}

Finally, we introduce the Meyers estimate. This lemma tells us that a solution of a uniformly elliptic equation has good regularity. We use the Meyers estimate in Lemma \ref{sec:EEDP} to obtain the estimates of $L^2$-norm of $\nab u$ in the boundary layer.
\begin{lem}[Global Meyers estimate]\label{lemlem:meyers}
Fix $p \in (2,\infty)$ and $0<\lam\leq\Lam<\infty$. Let $U\subseteq\mathbb{R}^d$ be a bounded Lipschitz domain and let $\aa$ be a measurable map from $U$ to the set of positive symmetric matrices with eigenvalues belonging to $[\lam,\Lam ]$. Suppose that $f \in W^{1,p}(U)$ and $u \in f +H_0^1(U)$ is the solution of 
\begin{equation}
-\nab \cdot \left(\aa(x)  \nab u\right) =0 \quad(\text{in } U), \qquad u=f \quad( \text{on } \partial U)
\end{equation}
in the distribution sense. Then, there exist constants $\del=\del(U,d,\lam,\Lam)>0$ and $C=C(U,d,\lam,\Lam)< \infty$ such that $u \in W^{1,(2+\del)\wedge p}(U)$ and
\begin{equation}
\| \nab u \|_{\ll^{(2+\del)\wedge p}(U)} \leq C\,\| \nab f \|_{\ll^{(2+\del)\wedge p}(U)}
\end{equation}
\end{lem}
\begin{proof}
See \cite[Theorem C.7]{MR3932093}.
\end{proof}

\section*{Acknowledgements} 
The author is very grateful to his superviser, Professor Seiichiro Kusuoka for his helping advice and encouragement.

\bibliographystyle{plain}
\bibliography{myref}

\end{document}